\documentclass[11 pt]{amsart}

\usepackage[english]{babel}
\usepackage[latin1]{inputenc}
\usepackage[T1]{fontenc}
\usepackage{graphicx}
\usepackage{amsmath}

\usepackage{amssymb}
\usepackage{latexsym}
\usepackage{amsthm}
\usepackage[all]{xy}
\usepackage{stmaryrd}
\usepackage{subfigure}
\usepackage{MnSymbol}
\usepackage{comment} 

\def\be{\begin{equation}}
\def\ee{\end{equation}}
\def\beq{\begin{eqnarray*}}
\def\eeq{\end{eqnarray*}}
\def\Z{\mathbb{Z}}

\newcommand{\pic}[3]{\parbox[c]{#1cm}{\includegraphics[scale=#2]{#3}}}
\newcommand{\picw}[2]{\parbox[c]{#1cm}{\includegraphics[width = #1cm]{#2}}}

\newtheorem{theo}{Theorem}[section]
\newtheorem{cor}[theo]{Corollary}
\newtheorem{lem}[theo]{Lemma}
\newtheorem{prop}[theo]{Proposition}

\theoremstyle{definition}
\newtheorem{defn}[theo]{Definition}
\newtheorem{ex}[theo]{Example}
\newtheorem{quest}[theo]{Question}
\newtheorem{rem}[theo]{Remark}

\author{Alessio Carrega}
\address{Dipartimento di Matematica, Largo Pontecorvo 5, 56127 Pisa, Italy}
\email{carrega at mail dot dm dot unipi dot it}

\title[The Tait conjecture in $S^1\times S^2$]{The Tait conjecture in $S^1\times S^2$}

\begin{document}

\begin{abstract}
The Tait conjecture states that reduced alternating diagrams of links in $S^3$ have the minimal number of crossings. It has been proved in 1987 by M. Thistlethwaite, L.H. Kauffman and K. Murasugi studying the Jones polynomial. In this paper we prove an analogous result for alternating links in $S^1\times S^2$ giving a complete answer to this problem. In $S^1\times S^2$ we find a dichotomy: the appropriate version of the statement is true for $\Z_2$-homologically trivial links, and our proof also uses the Jones polynomial. On the other hand, the statement is false for $\Z_2$-homologically non trivial links, for which the Jones
polynomial vanishes.
\end{abstract}

\maketitle

\setcounter{tocdepth}{1}
\tableofcontents

\section{Introduction}
One of the first invariants of links is the \emph{crossing number}: the minimal number of crossings that a diagram must have to present that link. In general it is hard to compute. During his attempt to tabulate all knots in $S^3$ in the $19^{\rm th}$ century \cite{Tait}, P.G. Tait stated three conjectures about crossing number, \emph{alternating links} and \emph{writhe number} (Definition~\ref{defn:alt_cr_num}). An \emph{alternating link} is a link that admits an \emph{alternating diagram}: a diagram $D$ such that the parametrization of its components $S^1\rightarrow D\subset D^2$ meets overpasses and underpasses alternately. All the conjectures have been proved before 1991. A diagram $D$ of a link in $S^3$ is said to be \emph{reduced} if it has no crossings as the ones in Fig.~\ref{figure:reducedDS3} (the blue parts cover the rest of the diagram). 

\begin{figure}[htbp]
\begin{center}
\includegraphics[scale=0.6]{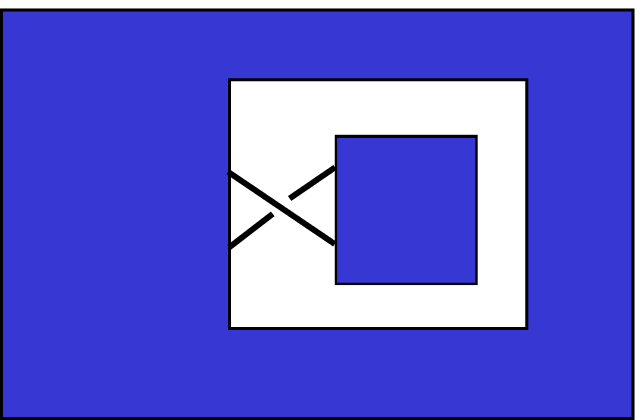}
\hspace{0.5cm}
\includegraphics[scale=0.6]{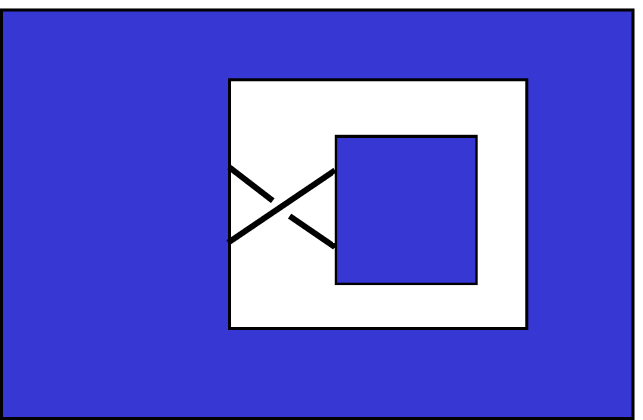}
\end{center}
\caption{Non reduced diagrams of links in $S^3$.}
\label{figure:reducedDS3}
\end{figure}

We focus just on one Tait conjecture, the following one: \emph{``Every reduced alternating diagram of links in $S^3$ has the minimal number of crossings''.} 

This conjecture was proved in 1987 by M. Thistlethwaite \cite{Thistlethwaite}, L.H. Kauffman \cite{Kauffman_Tait} and K. Murasugi \cite{Murasugi1, Murasugi2}. We are interested in the more general case of links in $S^1\times S^2$. A link diagram in the annulus $A=S^1\times [-1,1]$ defines a link in its thickening that is the solid torus $S^1\times D^2\cong A\times [-1,1]$, and hence on its double $S^1\times S^2$. The decomposition of $S^1\times S^2$ in two tori is unique up to isotopies (Theorem~\ref{theorem:Heegaard_split_S1xS2}), but the proper embedding of the annulus in a solid torus is not. Once fixed one such embedding of the annulus, every link in $S^1\times S^2$ can be described by a diagram on it, thus we can still talk about crossing number and alternating diagrams (Definition~\ref{defn:alt_cr_num}). We define the \emph{crossing number} of a link $L\subset S^1\times S^2$ as the minimal number of crossings that a diagram must have to represent $L$ in some embedded annulus. We call \emph{H-decomposition} the decomposition of $S^1\times S^2$ in two solid tori. We extend the notion of ``reduced'' as follows:
\begin{defn}\label{defn:reduced}
A link diagram $D\subset S^1\times [-1,1]$ is said to be \emph{simple} if it has no crossings as the ones shown in Fig.~\ref{figure:reducedD} and Fig.~\ref{figure:reducedD2} (in those figures the blue parts cover what remains of the diagram). 
\end{defn}

\begin{figure}[htbp]
\begin{center}
\includegraphics[scale=0.55]{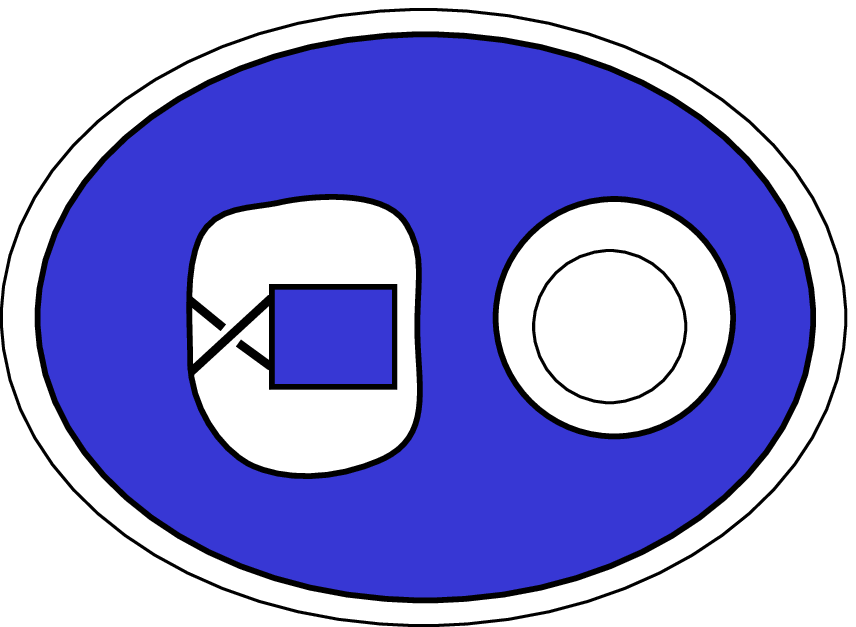} 
\hspace{0.5cm}
\includegraphics[scale=0.55]{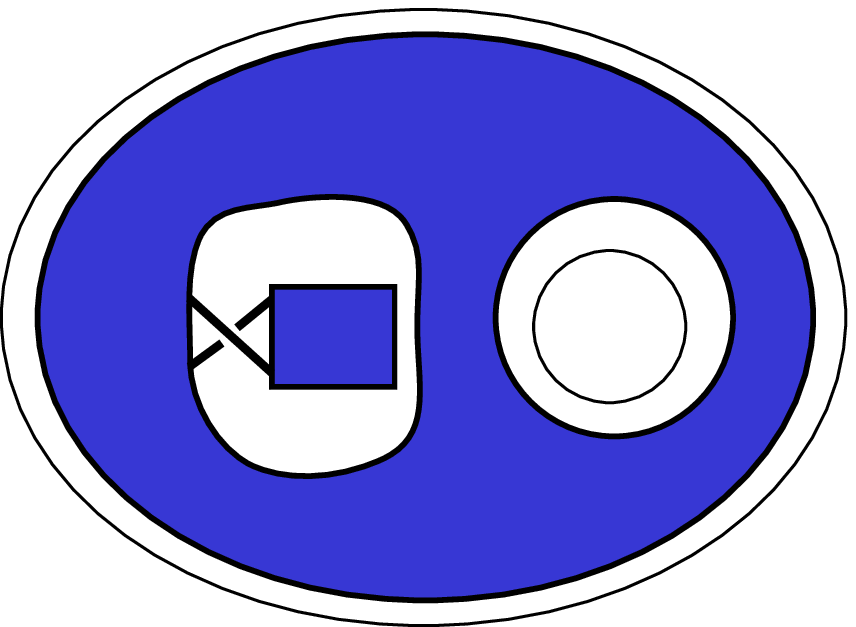}
\end{center}
\caption{Some non simple diagrams.}
\label{figure:reducedD}
\end{figure}

\begin{figure}[htbp]
\begin{center}
\includegraphics[scale=0.55]{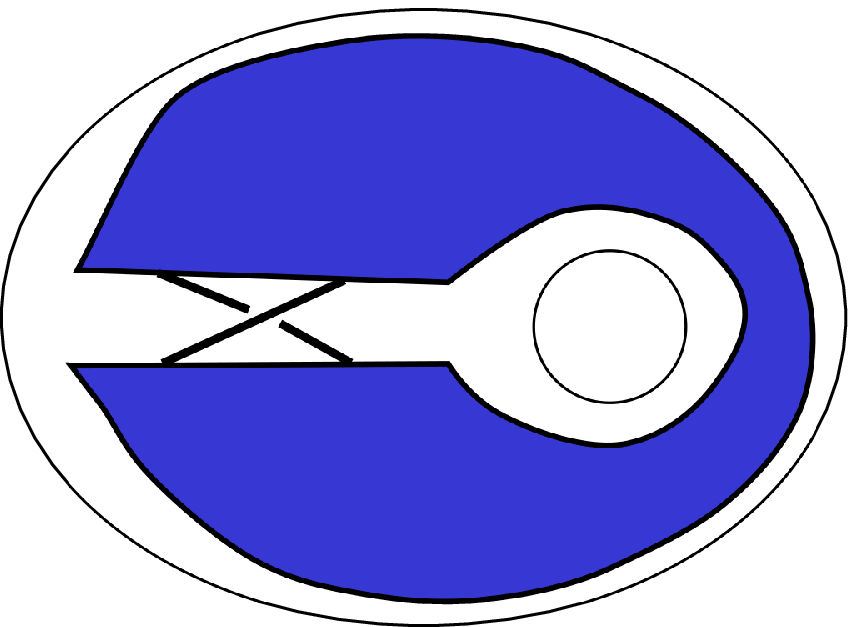} 
\hspace{0.5cm}
\includegraphics[scale=0.55]{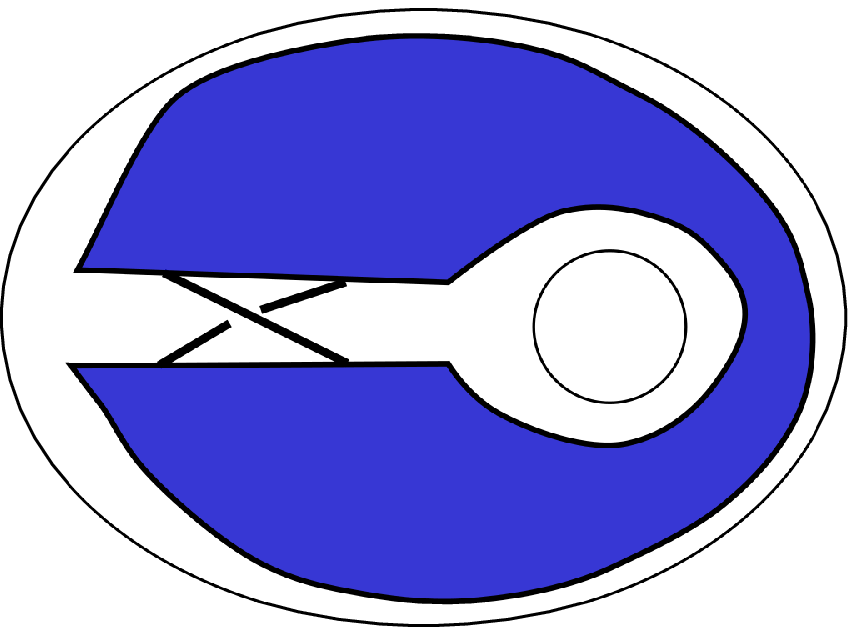}
\end{center}
\caption{More non simple diagrams.}
\label{figure:reducedD2}
\end{figure}

Our definition of ``simple'' was designed to obtain the following theorem, which is the main result of the paper. Before stating it, we specify that we say that a link $L$ is $\Z_2$-\emph{homologically trivial} if its homology class $[L]\in H_1(S^1\times S^2; \Z_2)$ is trivial $[L]=0$ ($\Z_2 =\Z / 2 \Z$). Furthermore we remind that being $\Z_2$-homologically trivial is equivalent to be the boundary of an embedded surface. 
\begin{theo}\label{theorem:Tait_conj}
Fix a proper embedding of the annulus in a solid torus of the H-decomposition of $S^1\times S^2$. Let $D\subset S^1\times [-1,1]$ be a $n$-crossing link diagram of a $\Z_2$-homologically trivial link $L \subset S^1\times S^2$. If $D$ is alternating and simple, and $L$ is non H-split (\emph{e.g.} a knot) and not contained in a 3-ball, then the number of crossings of $D$ is equal to the crossing number of $L$.
\end{theo}

We remark that the theorem says that $D$ has the minimal number of crossings among all diagrams that represent the link via some embedded annulus. Being \emph{non H-split} means that every diagram $D\subset S^1\times [-1,1]$ of $L$ is connected for every choice of the embedding of the annulus, namely it is so as a 4-valent graph (see Definition~\ref{defn:split}). 

The condition of being $\Z_2$-homologically trivial is essential. In fact we show that without that hypothesis the statement is false: once an embedding of the annulus is fixed, the diagrams in Fig.~\ref{figure:no_Tait} represent the same knot $K\subset S^1\times S^2$. The knot $K$ is $\Z_2$-homologically non trivial, and the diagrams are both alternating and simple, but they have different number of crossings. 

\begin{figure}
\begin{center}
\includegraphics[scale=0.55]{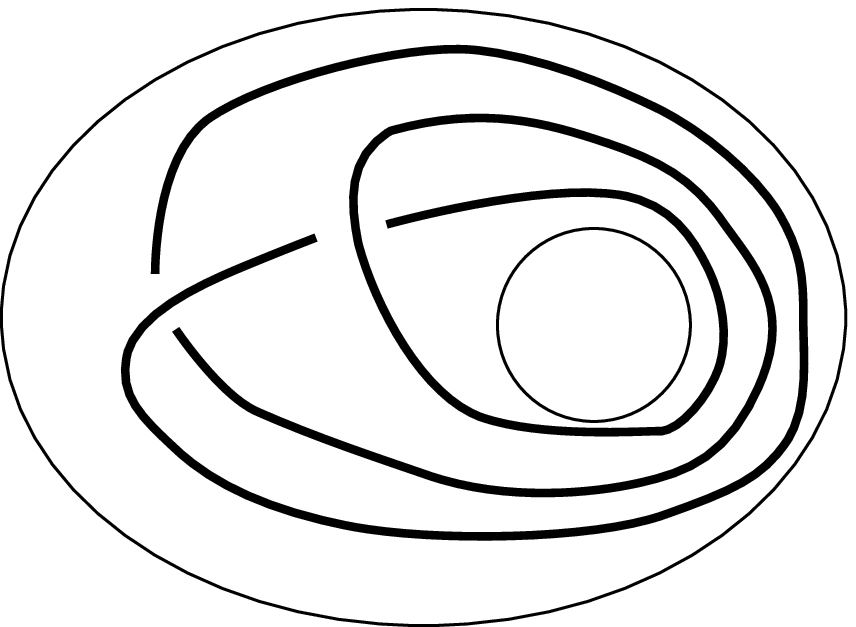} 
\hspace{0.5cm}
\includegraphics[scale=0.55]{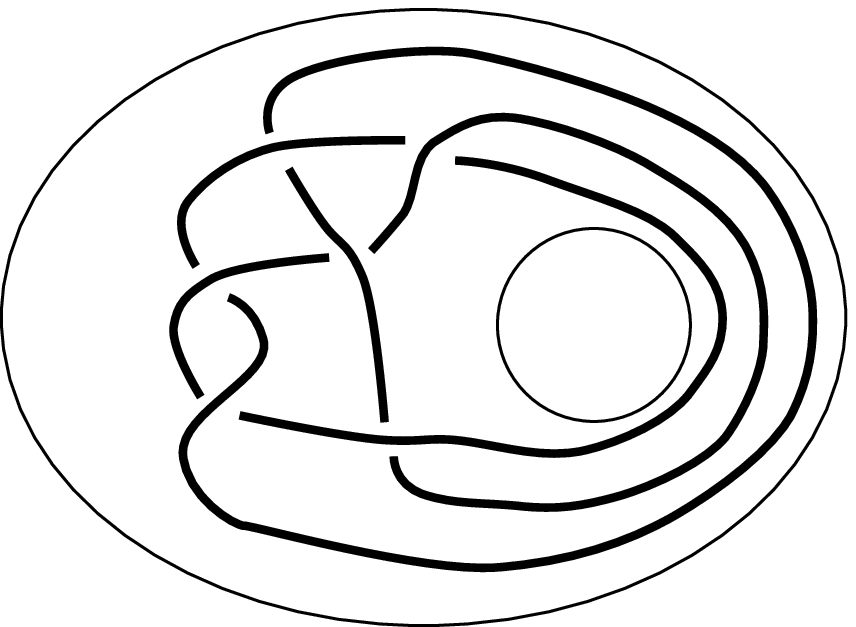}
\end{center}
\caption{Two alternating and simple diagrams of the same $\Z_2$-homologically non trivial knot which have different numbers of crossings.}
\label{figure:no_Tait}
\end{figure}

The proof of the classical case of links in $S^3$ is based on another famous invariant of links: the \emph{Jones polynomial}. This is an invariant that associates to each oriented link $L$ a Laurent polynomial with integer coefficients $J(L)\in \Z[A,A^{-1}]$ (we use the variable $A= t^{-\frac 1 4} = \sqrt{\pm q^{\pm 1}}$ and we normalize it so that $J\left( \pic{0.8}{0.3}{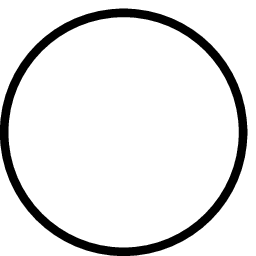} \right) = -A^2-A^{-2}$). This is related to an invariant of framed links: the \emph{Kauffman bracket} $\langle L \rangle$. They differ just by the multiplication of a power of $-A^3$. The Jones polynomial is the most simple \emph{quantum invariant}. One of the main targets of modern knot theory is to ``understand'' these invariants. There are few topological applications of the quantum invariants, and in particular of the Jones polynomial, and most of them are just conjectures. The theorem of Thistelwaithe-Kauffman-Murasugi is one of the most notable applications of the Jones polynomial. The \emph{breadth} $B(f)\in \Z$ is an integer associated to each Laurent polynomial $f$ (see Definition~\ref{def:breadth}). The breadth of the Jones polynomial, or of the Kauffman bracket, is invariant under the change of orientation and framing, and we have $B(J(L))=B(\langle L\rangle)$. The theorem of Thistelwaithe-Kauffman-Murasugi shows that we can get information about the crossing number from the \emph{breadth} of the Jones polynomial (or of the Kauffman bracket). In particular we can easily compute it if the link is alternating. 

The Jones polynomial, or more precisely the Kauffman bracket,  is also defined in $S^1\times S^2$ (see for instance \cite{Carrega-Martelli}, \cite{Costantino2} or Section~\ref{section:Preliminaries}). Proceeding in a similar way than Thistelwaithe-Kauffman-Murasugi, we get the following result:
\begin{theo}\label{theorem:Tait_conj_Jones}
Let $D\subset S^1\times [-1,1]$ be a connected, $n$-crossing link diagram of a $\Z_2$-homologically trivial link $L \subset S^1\times S^2$. If $D$ is alternating and without crossings as the ones in Fig.~\ref{figure:reducedD}, then
$$
B(\langle L\rangle) = \begin{cases}
4n+4 & \text{if $D$ is contained in a 2-disk} \\
4n-4k & \text{otherwise} 
\end{cases} ,
$$
where $k$ is the number of crossings as the ones in Fig.~\ref{figure:reducedD2}.
\end{theo}

We say that a link in $S^1\times S^2$ is alternating if it is represented by an alternating diagram in some embedded annulus. The previous theorem gives criteria to detect if a link in $S^1\times S^2$ is not alternating (Corollary~\ref{cor:conj_Tait_Jones}), and we show some examples (Example~\ref{ex:no_alt}).

The Kauffman bracket is also very sensitive of the $\Z_2$-homology class of the link. In fact the Jones polynomial of $\Z_2$-homologically non trivial links is always $0$ (Proposition~\ref{prop:0Kauff}).

In $S^3$ all link diagrams with the minimal number of crossings are reduced. We are not able to prove that all links in $S^1\times S^2$ that are different from the knot with crossing number $1$ (see Fig.~\ref{figure:rem_adeq}-(left)), have a simple diagram with the minimal number of crossings (probably it is not true). That is why we use the word ``simple'' instead of ``reduced''. However we can say that almost all of them have such diagram, and every diagram can be replaced by a \emph{quasi-simple} diagram with less crossings (see Subsection~\ref{subsec:non_red}):
\begin{defn}\label{defn:quasi-reduced}
A \emph{quasi-simple} diagram is a link diagram $D\subset S^1\times [-1,1]$ without crossings as the ones in Fig.~\ref{figure:reducedD} and with at most one crossing as the ones in Fig.~\ref{figure:reducedD2}.
\end{defn}
Clearly if a link does not intersect twice a non separating 2-sphere, it has a simple diagram. The statement of Theorem~\ref{theorem:Tait_conj} does not hold for all quasi-simple diagrams. In fact we show that the alternating diagrams in Fig.~\ref{figure:no_Tait_qr} represent the same $\Z_2$-homologically trivial knot once fixed the proper embedding of the annulus.

\begin{figure}[htbp]
\begin{center}
\includegraphics[scale=0.55]{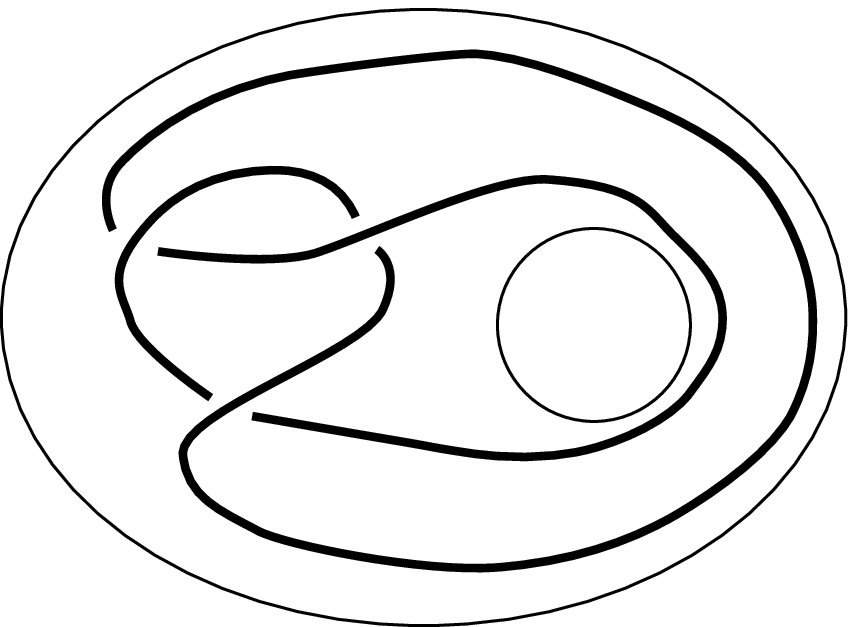} 
\hspace{0.5cm}
\includegraphics[scale=0.55]{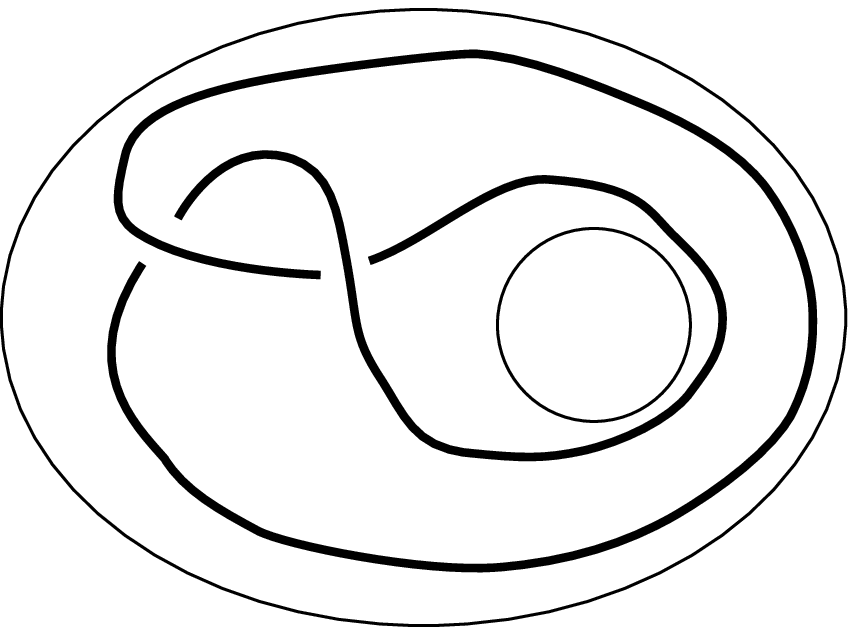}
\end{center}
\caption{Two quasi-simple alternating diagrams with different numbers of crossings and representing the same $\Z_2$-homologically trivial knot.}
\label{figure:no_Tait_qr}
\end{figure}

\subsection*{Structure of the paper}

In the second section we give the needed preliminaries about the Jones polynomial, the Kauffman bracket, \emph{skein theory} and diagrams.

In the third section we introduce some further notions, give some results, and we show that the natural extension of the conjecture is false if we remove the hypothesis of being $\Z_2$-homologically trivial or if we substitute ``simple'' with ``quasi-simple''.

In the fourth one we prove Theorem~\ref{theorem:Tait_conj} and Theorem~\ref{theorem:Tait_conj_Jones}. The proof follows the classical one for links in $S^3$.

In the last section the make some conjectures and recall all the questions.

\subsection*{Acknowledgments}
The author is warmly grateful to Bruno Martelli for his constant support and encouragement.

\section{Preliminaries}\label{section:Preliminaries}
In this section we recall some basic notions about the \emph{Jones polynomial}, the \emph{Kauffman bracket}, \emph{skein theory} and \emph{diagrams}.

\subsection{Kauffman bracket}
The \emph{Kauffman bracket} $\langle D \rangle \in \Z[A,A^{-1}]$ is a Laurent polynomial associated to a diagram $D$ of a link in $S^3$ or to a framed link. The Kauffman bracket is useful to define and understand the \emph{Jones polynomial}, they differ just by the multiplication of a power of $-A^3$. Thanks to result of Hoste-Przytycki \cite{HP2, Pr2} and (with different techniques) to Costantino \cite{Costantino2}, it is also defined in the connected sum $\#_g(S^1\times S^2)$ of copies of $S^1 \times S^2$. We briefly recall its definition.

Let $M$ be an oriented 3-manifold. Consider the field $\mathbb{Q}(A)$ of all rational functions with variable $A$ and coefficients in $\mathbb{Q}$. Let $V$ be the abstract $\mathbb{Q}(A)$-vector space generated by all framed links in $M$, considered up to isotopy. 
\begin{defn}
The \emph{skein vector space} $K(M)$ is the quotient of $V$ by all the possible \emph{skein relations}: 
$$
\begin{array}{rcl}
 \pic{1.2}{0.3}{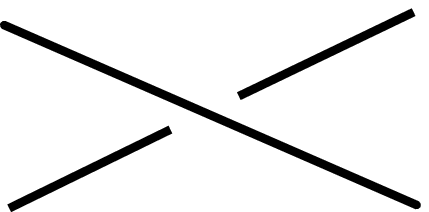}  & = & A \pic{1.2}{0.3}{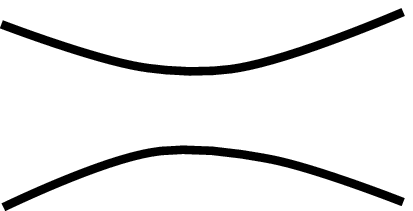}  + A^{-1}  \pic{1.2}{0.3}{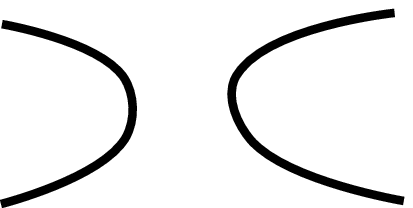}  \\
 L \sqcup \pic{0.8}{0.3}{banp.eps}  & = & (-A^2 - A^{-2})  D  \\
 \pic{0.8}{0.3}{banp.eps}  & = & (-A^2-A^{-2}) \varnothing
\end{array}
$$
These are local relations where the framed links in an equation differ just in the pictured 3-ball that is equipped with a positive trivialization. An element of $K(M)$ is called a \emph{skein} or a \emph{skein element}.
\end{defn}

Since the work of Kauffman we have that the skein vector space of $S^3$ is isomorphic to the base field and is generated by the empty set $\varnothing$. Furthermore the skein element of each link $L$ is equal to the empty set multiplied by the Kauffman bracket of $L$, that is $L= \langle L\rangle \cdot \varnothing$. More in general we have:
\begin{theo}[Hoste, Przytycki]
The skein vector space $K(\#_g(S^1\times S^2))$ of $\#_g(S^1\times S^2)$ is isomorphic to $\mathbb{Q}(A)$ and generated by the empty skein $\varnothing$.
\begin{proof}
This is due to Hoste and Przytycki \cite{HP2, Pr1, Pr2}, see also \cite[Proposition 1]{BFK}.
\end{proof}
\end{theo}

\begin{defn}
A framed link $L\subset \#_g(S^1\times S^2)$ determines a skein $L\in K(\#_g(S^1\times S^2))$ and as such it is equivalent to $\langle L \rangle \cdot \varnothing$ for a unique coefficient $\langle L \rangle \in \mathbb{Q}(A)$. This coefficient is by definition the \emph{Kauffman bracket} $\langle L \rangle$ of $L$.
\end{defn}

\begin{rem}\label{rem:tensor}
There is an obvious canonical linear map $K(M) \to K(M\# N)$ defined by considering a skein in $M$ inside $M\# N$. The linear map $K(\#_g(S^1\times S^2)) \mapsto K(\#_{g+1}(S^1\times S^2))$ sends $\varnothing$ to $\varnothing$ and hence preserves the bracket $\langle L \rangle$ of an $L\subset \#_g(S^1\times S^2)$.

This shows in particular that if $L$ is contained in a 3-ball, the bracket $\langle L \rangle$ is the same that we would obtain by considering $L$ inside $S^3$.
\end{rem}

The Kauffman bracket is not an invariant of links without framing. In fact it changes under Reidemeister moves of the first type: $\left\langle \pic{1.2}{0.3}{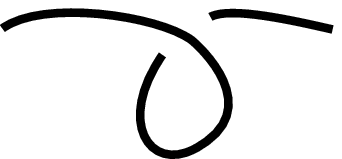} \right\rangle  =  -A^3 \left\langle \pic{1.2}{0.3}{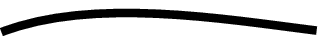} \right\rangle$. Two framings differ by adding either some positive twists, or some negative ones, they are equivalent to applying Reidemeister moves of the first type (see Fig.~\ref{figure:framing_change}). Thus if we change the framing of a link, its Kauffman bracket changes just by a multiplication of $-A^3$.

\begin{figure}[htbp]
$$
\pic{1.8}{0.4}{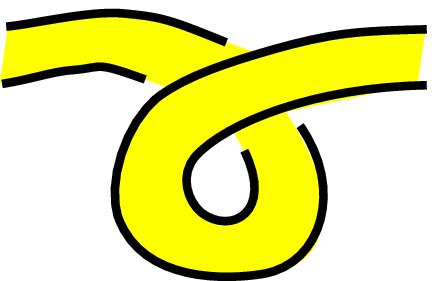} \leftrightarrow \pic{1.8}{0.4}{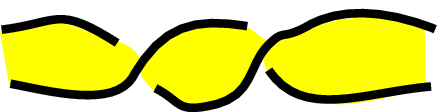} \leftrightarrow \pic{1.8}{0.4}{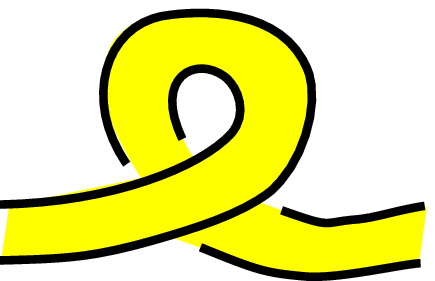} 
$$
\caption{A negative twist.}
\label{figure:framing_change}
\end{figure} 

There is a standard way to color the components of a framed link $L\subset M$ with a natural number and get a skein element of $K(M)$. These colorings are based on particular elements of the \emph{Temperley-Lieb algebra} called \emph{Jones-Wenzl projectors}. Coloring a component with $0$ is equivalent to remove it, while coloring with $1$ corresponds to consider the standard skein. Usually a component colored with $n$ is pictured with a square box with the letter ``$n$'' over the component, or over $n$ parallel copies of the component.

\subsection{Diagrams and moves}\label{subsec:diag}
It is well known that two link diagrams in $D^2$ represent the same link in $S^3$ if and only if they are related by Reidemeister moves. Since $S^1\times S^2$ is the double of a solid torus, we have that, given one such decomposition of $S^1\times S^2$, every link there can be isotoped in such solid torus. We call one such decomposition of $S^1\times S^2$ a \emph{H-decomposition}. The following theorem ensures that one such decomposition is unique:
\begin{theo}\label{theorem:Heegaard_split_S1xS2}
Every two embeddings of a torus in $S^1\times S^2$ that split it into two solid tori are isotopic.
\begin{proof}
See for instance the proof of \cite[Theorem 2.5]{Hatcher}.
\end{proof}
\end{theo}

The solid torus $V$ is the natural 3-dimensional thickening $V\cong A\times [-1,1]$ of the annulus $A= S^1\times [-1,1]$. Hence, once a proper embedding of the annulus in the solid torus is fixed, every link in $S^1\times S^2$ can be represented by a link diagram in the annulus $S^1\times[-1,1]$. The embedding of the annulus is not unique up to isotopies, two of them differ by the application of some twists.

Clearly Reidemesister moves still do not change the represented link, but they are not sufficient to connect all the diagrams representing the same link. 

Now we describe a new move (see Fig.~\ref{figure:new_move}) that will be used later to prove that all the hypothesis of Theorem~\ref{theorem:Tait_conj} are needed, namely that the pairs of diagrams in Fig.~\ref{figure:no_Tait}  and Fig.~\ref{figure:no_Tait_qr} represent the same knot. The move is essentially the second Kirby move. Given a diagram $D$ in the annulus and a proper embedding of the annulus in the solid torus $V$, we get a position (an embedding not up to isotopy) of the link $L\subset V \subset S^1\times S^2$ described by $D$. We embed the solid torus in $\mathbb{R}^3$ in the standard way so that the image of the embedded annulus lies on $\mathbb{R}^2\subset \mathbb{R}^3$. Then we add a 0-framed meridian of the solid torus. We have obtained a surgery presentation of the pair $(L,S^1\times S^2)$ in $\mathbb{R}^3\subset S^3$ that is in regular position with respect to $\mathbb{R}^2\subset \mathbb{R}^3$. We apply the second Kirby move to a component of $L$ along the 0-framed meridian along an obious band. This gives another surgery presentation of $(L,S^1\times S^2)$ consisting of a link $L'$ in the solid torus encircled by the 0-framed meridian. The link $L'$ is still in regular position and hence gives another diagram of $L$ in the annulus.

\begin{figure}[htbp]
$$
\picw{4.8}{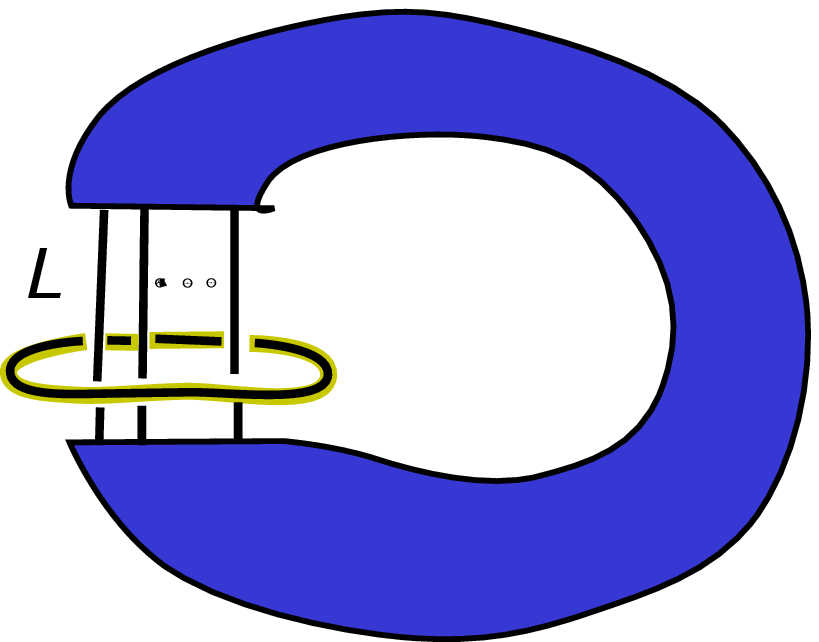} \leftrightarrow \picw{4.8}{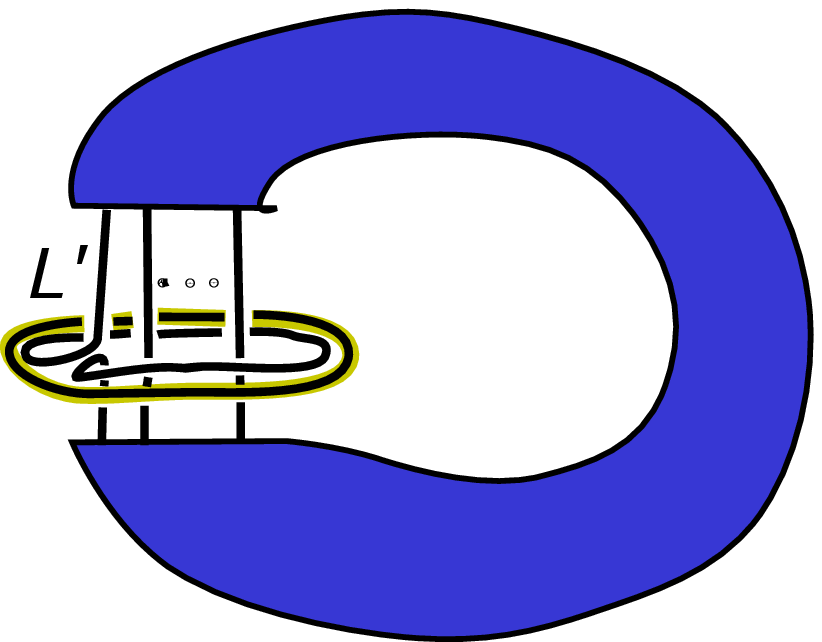}
$$
\caption{A new move on diagrams in the annulus. The links $L$ and $L'$ represent two different links in the solid torus but they represent the same in $S^1\times S^2$. The highlighted circle is the 0-framed meridian needed to get a surgery presentation of $(L,S^1\times S^2)$. The highlighted circle is not contained in the solid torus, hence the diagrams in the annulus do not contain its projection.}
\label{figure:new_move}
\end{figure}

\begin{quest}\label{quest:reid_moves}
Once fixed the proper embedding of the annulus, are Reidemeister moves together with the move described above sufficient to connect all the diagrams in the annulus representing the same link in $S^1\times S^2$?
\end{quest}

Once a proper embedding of the annulus in a solid torus of the H-decomposition is fixed and given a diagram $D\subset S^1\times [-1,1]$ of a link $L\subset S^1\times S^2$, we can get a diagram $D'\subset S^1\times [-1,1]$ that represents $L$ with the embedding of $S^1\times [-1,1]$ obtained from the previous one by adding a twist following the move described in Fig.~\ref{figure:twist_diagr}.

\begin{figure}[htbp]
\beq
\picw{4.8}{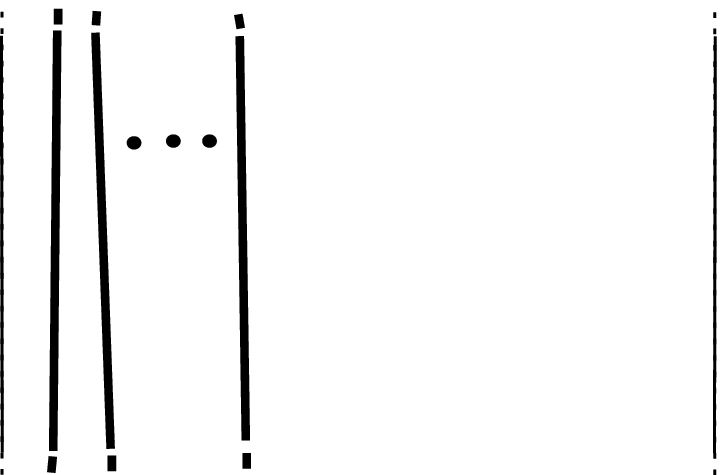} & \longrightarrow & \picw{4.8}{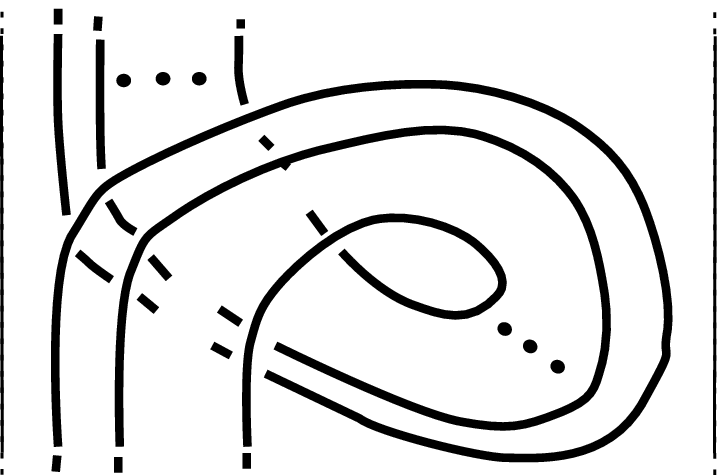} \\
D & & D' 
\eeq
\caption{Two diagrams of the same link in $S^1\times S^2$, the embedding of the annulus for $D'$ differs from the one of $D$ by the application of a positive twist. The diagrams differ just in the pictured portion that is diffeomorphic to $[-1,1]\times (-1,1)$.}
\label{figure:twist_diagr}
\end{figure}

\begin{defn}\label{defn:alt_cr_num}
A link diagram $D\subset S^1\times [-1,1]$ is \emph{alternating} if the parametrization of its components $S^1\rightarrow D\subset S^1\times [-1,1]$ meets overpasses and underpasses alternately.

Let $L$ be a link in $S^1\times S^2$. The link $L$ is \emph{alternating} if there is an alternating diagram $D\subset S^1\times [-1,1]$ that represents $L$ with a proper embedding of the annulus in a solid torus of the H-decomposition. The \emph{crossing number} of $L$ is the minimal number of crossings that a link diagram $D\subset S^1\times [-1,1]$ must have to represent $L$ with some embedding of the annulus.
\end{defn}

\begin{rem}
Let $\varphi: S^1\times S^2\rightarrow S^1\times S^2$ be a diffeomorphism and let $L\subset S^1\times S^2$ be a link with a fixed position ($L$ is just a sub-manifold, it is not up to isotopies). Suppose that $L$ is in regular position for a proper embedded annulus $A\subset V_1 \subset S^1\times S^2$ in a solid torus of the H-decomposition $S^1\times S^2= V_1\cup V_2$, $V_1\cong V_2\cong S^1\times D^2$. Hence the couple $(L,A)$ defines a link diagram $D\subset S^1\times [-1,1]$. Then the link $\varphi(L)$ is in regular position for the annulus $\varphi(A)$ that is proper embedded in $\varphi(V_1)$. By Theorem~\ref{theorem:Heegaard_split_S1xS2} $\varphi(V_1) = V_j$ (up to isotopy) for some $j=1,2$. The couple $(\varphi(L), \varphi(A))$ still defines the diagram $D\subset S^1\times [-1,1]$. Therefore the crossing number and the condition of being alternating are invariant under diffeomorphisms of $S^1\times S^2$.
\end{rem}

\section{The extended Tait conjecture}
In this section we explain why a natural extension of the Tait conjecture is false for $\Z_2$-\emph{homologically non trivial} links, we show that all links have a \emph{quasi-simple} diagram in the annulus, we show that we can not simply replace ``simple'' with ``quasi-simple'' in Theorem~\ref{theorem:Tait_conj}, and we give other notions and results in order to state and prove the main theorems.

\subsection{Alternating diagrams for $\Z_2$-homologically non trivial links}
\begin{prop}\label{prop:no_Tait}
Once an embedding of the annulus is fixed, two link diagrams in $S^1\times [-1,1]$ differing only in part of the annulus diffeomorphic to $(-1,1)\times [-1,1]$ where they are of the form
$$
\pic{1.8}{0.4}{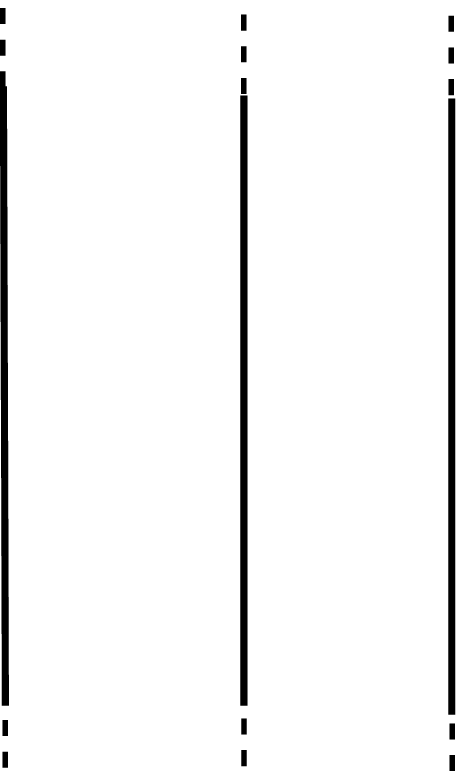} \ \ \ \text{and} \ \ \ \pic{1.8}{0.4}{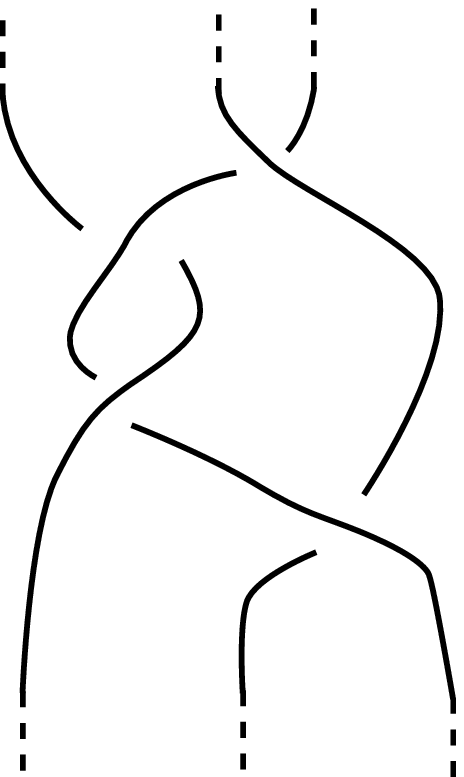},
$$
represent the same link in $S^1\times S^2$.
\begin{proof}
Let us consider the diagram with three parallel strands. If we apply the second Kirby move on the left strand of the initial diagram over a 0-framed unknot encircling the three strands (the move described in Subsection~\ref{subsec:diag}), we get the second diagram of the sequence in Fig.~\ref{figure:prop_no_Tait} or its mirror image. To get the third diagram of Fig.~\ref{figure:prop_no_Tait} we apply some Reidemeister moves. The fourth one is obtained applying again the previous moves to the third diagram but on the opposite sense. The other equivalences of Fig.~\ref{figure:prop_no_Tait} come applying Reidemeister moves.

\end{proof}
\end{prop}

\begin{figure}
\beq
 & \pic{2.2}{0.5}{diagrp1.eps}  \ \  \leftrightarrow\ \  \pic{2.6}{0.5}{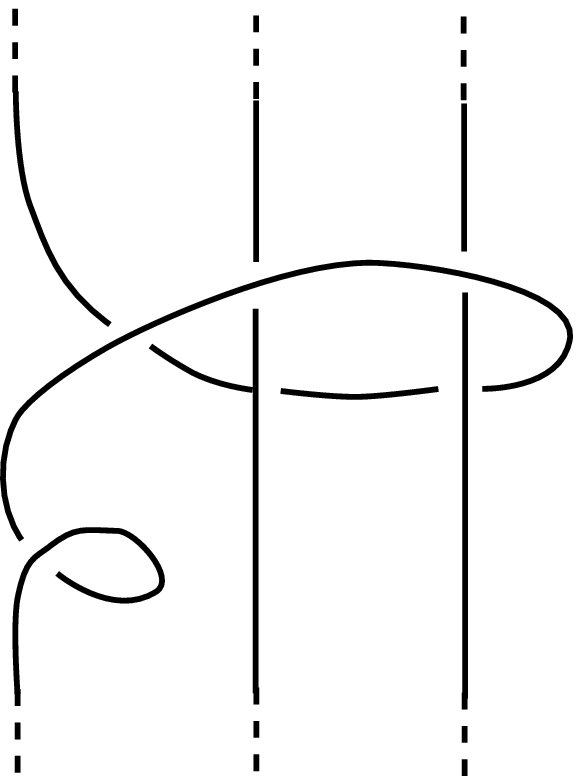} \ \  \leftrightarrow\ \  \pic{2.6}{0.5}{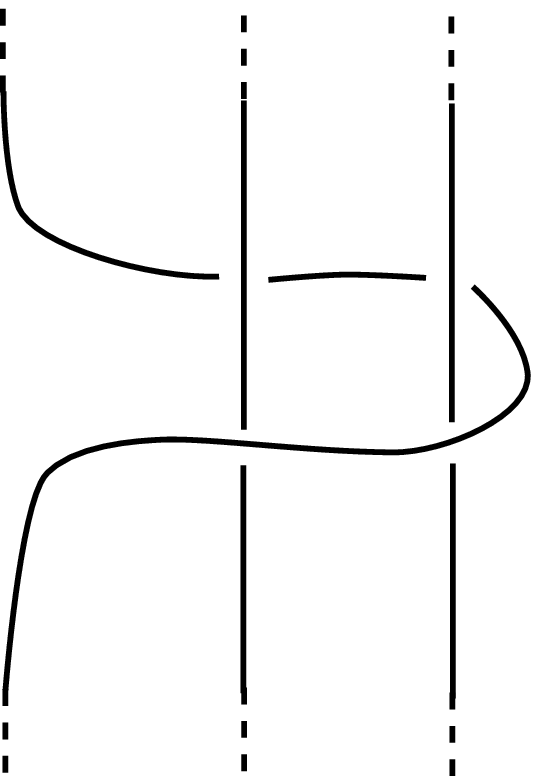} & \\
 & \leftrightarrow\ \  \pic{2.6}{0.5}{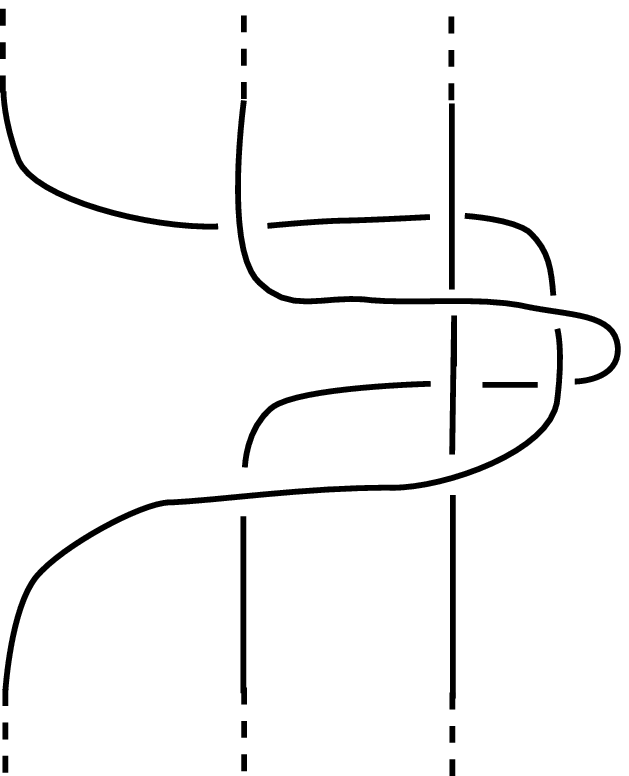}  \ \  \leftrightarrow\ \  \pic{2.6}{0.5}{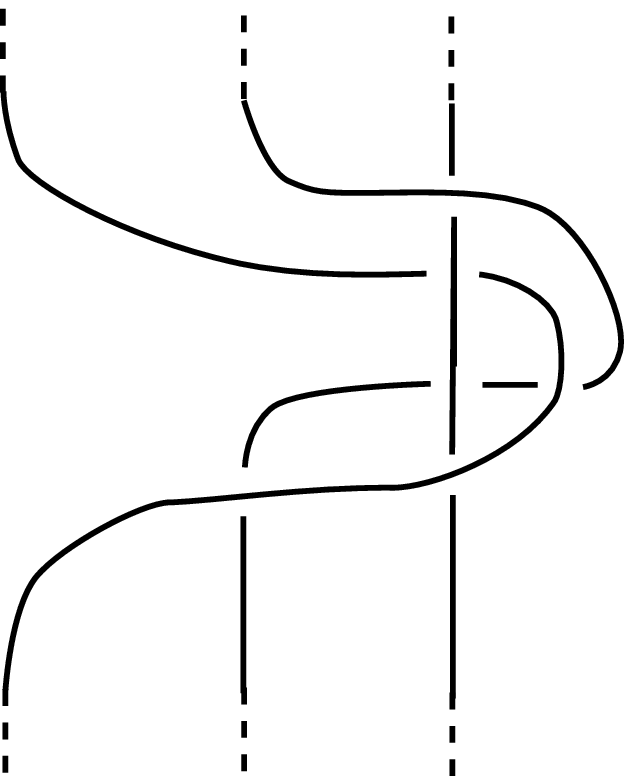}  \ \ \leftrightarrow \ \ \pic{2.6}{0.5}{diagrp2.eps}  . & 
\eeq
\caption{Moves on diagrams intersecting $(-1, 1)\times [-1,1]\subset S^1\times [-1,1]$ in three parallel strands.}
\label{figure:prop_no_Tait}
\end{figure}

Proposition~\ref{prop:no_Tait} implies that, once the embedding of the annulus is fixed, the diagrams in Fig.~\ref{figure:no_Tait} represent the same knot $K$ in $S^1\times S^2$. The knot $K$ is $\Z_2$-homologically non trivial. The diagrams are all alternating and simple but they have different numbers of crossings. Therefore the natural extension of the Tait conjecture for $\Z_2$-homologically non trivial links in $S^1\times S^2$ is false.

\subsection{Reducing diagrams}\label{subsec:non_red}
In $S^3$ every link diagram with the minimal number of crossings is reduced (a diagram without crossings as the ones of Fig.~\ref{figure:reducedDS3}). It would be nice if in $S^1\times S^2$ every link different from the knot with crossing number $1$ (see Fig.~\ref{figure:rem_adeq}-(left)) had a simple diagram (Definition~\ref{defn:reduced}) with the minimal number of crossings. Clearly every diagram $D\subset S^1\times [-1,1]$ can be replaced by another diagram $D'$ without crossings as the ones of Fig.~\ref{figure:reducedD}, with no more crossings ($n(D') \leq n(D)$, where $n(D)$ is the number of crossings of $D$) and representing the same link in $S^1\times S^2$ by the same proper embedding of the annulus in a solid torus of the H-decomposition. Therefore if the link does not intersect twice a non separating 2-sphere, it has a simple diagram with the minimal number of crossings, thus almost all links have one such diagram. Applying some Reidemeister moves we can put all the crossings as the ones in Fig.~\ref{figure:reducedD2} in the same band, one after the other (see the first move in Fig.~\ref{figure:screw}). We can suppose that these crossings are all of the same type: the band is represented by an alternating part of diagram (see the second move of Fig.~\ref{figure:screw}). If we apply the move described in Subsection~\ref{subsec:diag} we add or remove two crossings as the ones in Fig.~\ref{figure:reducedD2} (see the last three moves in Fig.~\ref{figure:screw}). Therefore, given a $n$-crossing diagram $D$ with $k$ crossings as the ones in Fig.~\ref{figure:reducedD2}, we can get another diagram $D'$ representing the same link which has $n- k + \bar k$ crossings, where $\bar k =0$ if $k\in 2\Z$ and $\bar k =1$ if $k \in 2\Z+1$, and has $\bar k$ crossings as the ones in Fig.~\ref{figure:reducedD2}. That means that every diagram of a link in $S^1\times S^2$ can be replaced by a quasi-simple diagram (Definition~\ref{defn:quasi-reduced}) with less crossings.

\begin{quest}\label{quest:qred_no_red}
Are there links in $S^1\times S^2$ with crossing number bigger than $1$ and not admitting a simple diagram in the annulus with the minimal number of crossings? Is it true that links in $S^1\times S^2$ which are alternating, non H-split (Definition~\ref{defn:split}), intersecting twice a non separating 2-sphere, $\Z_2$-homologically trivial and not bounding an orientable surface, do not have a simple diagram in the annulus with the minimal number of crossings? Are these links the only ones not admitting such diagrams?
\end{quest}

\begin{figure}[htbp]
$$
\picw{1.5}{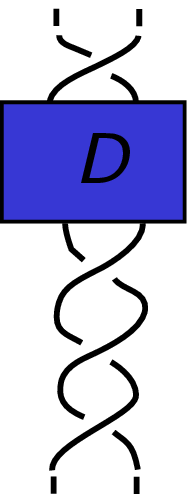} \ \ \leftrightarrow\ \  \picw{1.5}{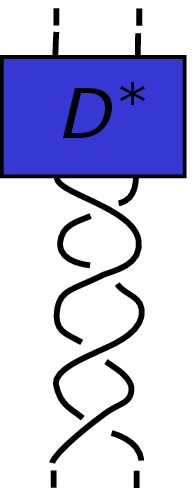} \ \ \leftrightarrow \ \ \picw{1.5}{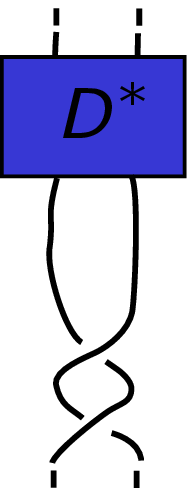}\ \  \leftrightarrow
$$
$$
\leftrightarrow\ \ \picw{1.5}{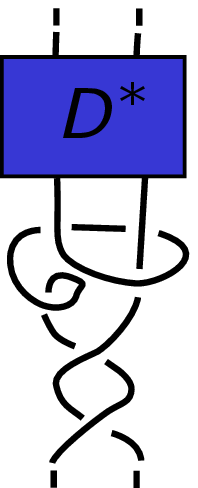}\ \  \leftrightarrow \ \ \picw{1.5}{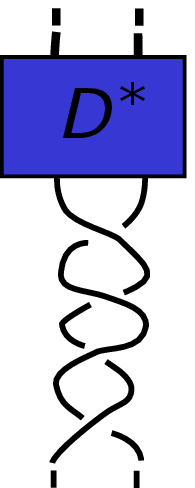} \ \ \leftrightarrow\ \  \picw{1.5}{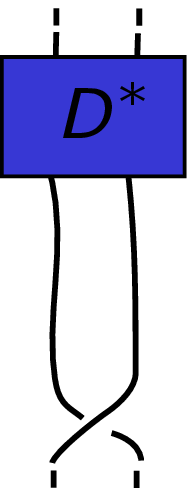}
$$
\caption{Moves for the crossings as the ones in Fig.~\ref{figure:reducedD2}. The part of diagram $D^*$ is the mirror image of $D$.}
\label{figure:screw}
\end{figure}

Applying the above moves we can substitute a crossing as the ones in Fig.~\ref{figure:reducedD2} with its mirror image and get another diagram representing the same link in $S^1\times S^2$. Therefore we can perform the moves in Fig.~\ref{figure:no_Tait_qr_moves} to get diagrams of the same link in $S^1\times S^2$. This shows that the diagrams in Fig.~\ref{figure:no_Tait_qr} represent the same knot. Therefore Theorem~\ref{theorem:Tait_conj} becomes false if we simply replace ``simple'' with ``quasi-simple''.

\begin{figure}[htbp]
$$
\picw{2}{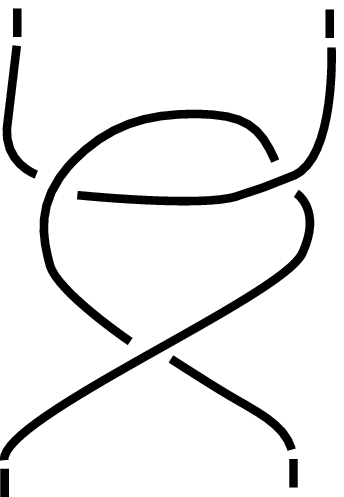} \ \ \leftrightarrow \ \  \picw{2}{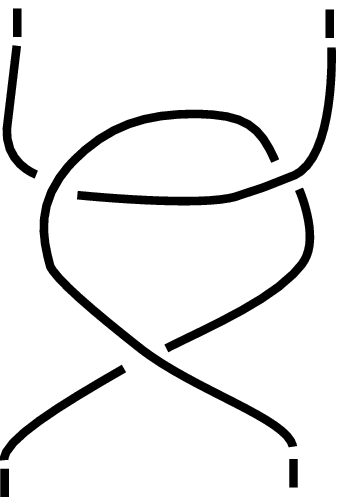} \ \ \leftrightarrow \ \ \picw{2}{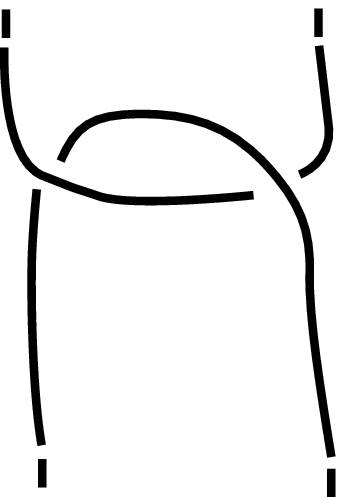}
$$
\caption{Moves in $(-1, 1)\times [-1,1] \subset S^1\times[-1,1]$ showing that the diagrams in Fig.~\ref{figure:no_Tait_qr} represent the same knot.}
\label{figure:no_Tait_qr_moves}
\end{figure}

\subsection{New notions and results}
The equalities in Fig.~\ref{figure:Cheb1} and Fig.~\ref{figure:sphere} are well known relations in skein theory. The one in Fig.~\ref{figure:Cheb1} takes place in the solid torus and holds for colored parallel copies of the core with trivial framing (the framing given by the embedded annulus where the copies lie) where $c \geq 2$. The ones in Fig.~\ref{figure:sphere} take place in a neighborhood of an embedded 2-sphere intersected once or twice by the link. It holds for every colors $i,j\geq 0$ and $d_{i,j}$ is the Kronecker delta: $d_{i,i}:=1$ and $d_{i,j}:=0$ for $i\neq j$.

\begin{figure}
\begin{center}
\includegraphics[width = 5.6 cm]{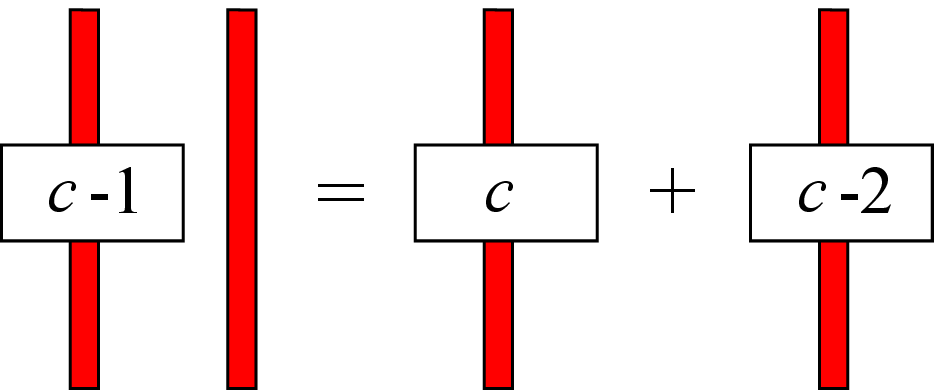}
\caption{A move in the skein space of the solid torus ($c \geq 2$).}
\label{figure:Cheb1}
\end{center}
\end{figure}

\begin{figure}
\begin{center}
\includegraphics[width = 10 cm]{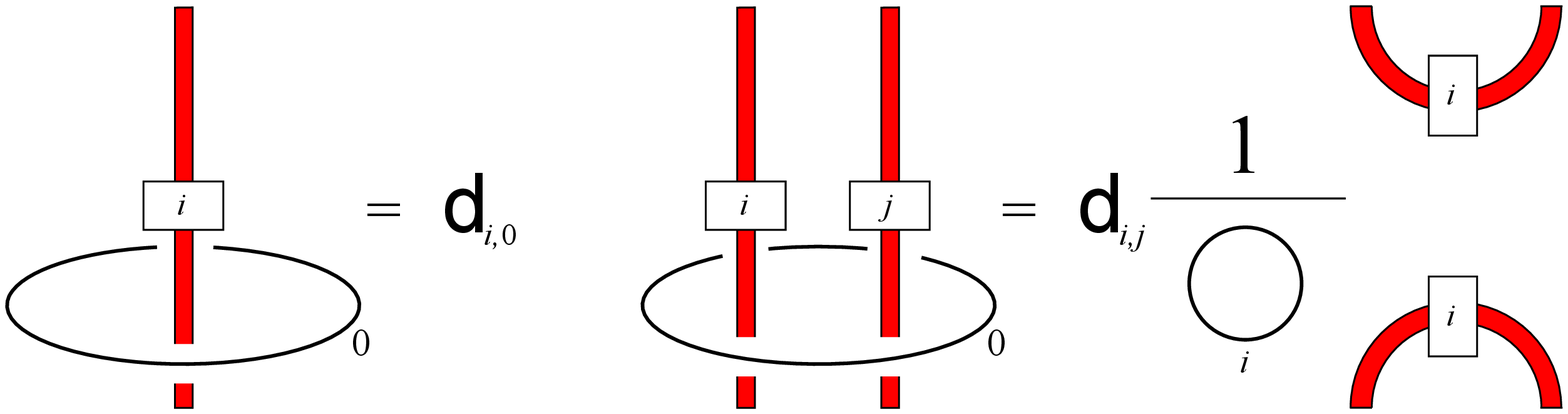}
\caption{Sphere intersection. The symbol $d_{i,j}$ is the Kronecker delta.}
\label{figure:sphere}
\end{center}
\end{figure}

\begin{lem}\label{lem:parallel_cores}
Let $L\subset S^1\times S^2$ be the framed link consisting of $k\geq 0$ parallel copies of the core $S^1\times \{x\}$ with the trivial framing (the blackboard framing given by a diagram without crossings). Then $\langle L \rangle =0$ if $k\in 2\Z+1$, otherwise it is a positive integer.
\begin{proof}
This proof is based on the description of a computation. In Fig.~\ref{figure:ex_lem} is shown as an example the case $k=3$. Let $K$ be the core of $S^1\times S^2$ with the trivial framing. In each step we get a linear combination with positive integers of framed links consisting of colored copies of $K$: in each framed link there is one copy with a non negative color, while the others are colored with $1$. Applying the equality of Fig.~\ref{figure:Cheb1} to each summand we fuse two components, one of them has color $1$ and the other one has the maximal color of that framed link. We apply this equality until we get a linear combination with positive integer coefficients of links consisting just of one colored copy of $K$. The colors of the final summands are all odd if $k\in 2\Z+1$, otherwise they are all even and the coefficient that multiplies the empty set (the copy colored with $0$) is non null. The equality of Fig.~\ref{figure:sphere}-(left) says that all the summands except the one with color $0$ are null.
\end{proof}
\end{lem}

\begin{figure}
\begin{center}
\includegraphics[width = 5.6 cm]{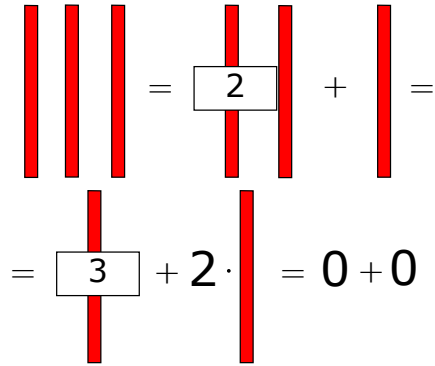}
\caption{The computation of the Kauffman bracket of three parallel copies of the core of $S^1\times S^2$.}
\label{figure:ex_lem}
\end{center}
\end{figure}

\begin{defn}
For every integer $k\geq 0$ we denote with $\alpha(k)$ the skein of $k$ parallel copies of the core of $S^1\times S^2$ with the trivial framing.
\end{defn}
By Lemma~\ref{lem:parallel_cores} $\alpha(k)$ is a positive integer for every even $k$, and $\alpha(k)=0$ for $k\in 2\Z+1$.

\begin{prop}
$$
\alpha(2n) = \frac{1}{n+1} \binom{2n}{n} .
$$
\begin{proof}
See \cite[Corollary 5]{HP2}.
\end{proof}
\end{prop}

\begin{defn}
Let $D$ be a link diagram in the annulus. A \emph{Kauffman state} of $D$ is a function $s$ from the set of crossings of $D$ to $\{1,-1\}$. The assignment of $\pm 1$ to a crossing determines a unique way to remove that crossing as described in Fig.~\ref{figure:splitting}. Hence a state removes all the crossings producing a finite collection of non intersecting circles in the surface. This collection  of circle is called the \emph{splitting}, or the \emph{resolution}, of $D$ with $s$. Let $sD$ be the number of homotopically trivial circles of the splitting of $D$ with $s$, and let $p(s)$ be the number of homotopically non trivial circles. Furthermore $\sum_i s(i)$ is the sum of all the signs associated to the crossings by $s$.
\end{defn}

\begin{rem}\label{rem:p(s)}
Given a state $s$, $p(s)$ is odd if and only if the link is $\Z_2$-homologically non trivial.
\end{rem}

\begin{figure}[htbp]
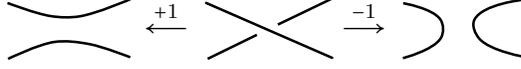

$$
\pic{1.6}{0.4}{Acanalep.eps} \ \stackrel{+1}{\longleftarrow} \ \pic{1.6}{0.4}{incrociop.eps} \ \stackrel{-1}{\longrightarrow} \ \pic{1.6}{0.4}{Bcanalep.eps}
$$
\caption{The splitting of a crossing.}
\label{figure:splitting}
\end{figure}

Proceeding by induction and splitting the crossings we get the following:
\begin{prop}\label{prop:state_sum}
Let $L$ be a framed link in $S^1\times S^2$ and $D\subset S^1\times[-1,1]$ be a diagram that represents $L$ with a fixed embedding of the annulus. Then
$$
\langle L \rangle = \sum_s \langle D \ |\ s \rangle  \ \in \Z[A, A^{- 1}] ,
$$
where the sum is taken over all the Kauffman states of $D$ and
$$
\langle D \ | \ s \rangle := \alpha( p(s) ) A^{\sum_i s(i)} (-A^2 -A^{-2})^{sD} .
$$
\end{prop}

\begin{prop}\label{prop:0Kauff}
Let $L$ be a framed link in $S^1\times S^2$. Suppose that the $\Z_2$-homology class of $L$ is non trivial
$$
0 \neq [L]  \in H_1(S^1\times S^2; \Z_2) .
$$
Then
$$
\langle L \rangle = 0 .
$$
\begin{proof}
Let $D\subset S^1\times [-1,1]$ be a diagram of $L$ for some embedding of the annulus. By Remark~\ref{rem:p(s)} for every state $s$ of $D$ the number $p(s)$ is odd. Therefore by Proposition~\ref{prop:state_sum} all the summands of $\langle L \rangle$ are null.
\end{proof}
\end{prop}

\begin{rem}
In $S^3$ the Kauffman bracket of every framed link $L\subset S^3$ is always non null. In fact we have that for such links the evaluation of the Kauffman bracket in $- 1$ is $\langle L\rangle|_{A=- 1} = (-2)^k$ where $k$ is the number of components of $L$. Therefore $\langle L \rangle$ can not be $0$. We can get the same result for homotopically trivial links in $S^1\times S^2$. By Proposition~\ref{prop:0Kauff}, it is natural to ask if the Kauffman bracket of a framed link in $S^1\times S^2$ is null if and only if the link is $\Z_2$-homologically non trivial. The answer to this question is no. In Fig.~\ref{figure:ex_0Kauf} is shown a $\Z_2$-homologically trivial knot in $S^1\times S^2$ whose Kauffman bracket is $0$. By Theorem~\ref{theorem:Tait_conj_Jones}, such $\Z_2$-homologically trivial links can not have a connected, simple and alternating diagrams in the annulus.
\end{rem}

\begin{figure}
\begin{center}
\includegraphics[scale=0.55]{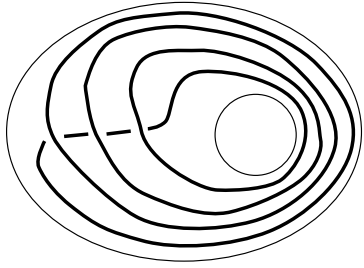} 
\end{center}
\caption{A $\Z_2$-homologically trivial knot in $S^1\times S^2$ whose Kauffman bracket is $0$.}
\label{figure:ex_0Kauf}
\end{figure}

\begin{defn}\label{def:breadth}
Let $f$ be a non null Laurent polynomial. We denote with $B(f)$ the difference between the highest degree of the non zero monomials in $f$ and the lowest one. If $f=0$ we define $B(f):=0$. We call it the \emph{breadth} of $f$.
\end{defn}
Note that the breadth of the Kauffman bracket of a link remains the same if we modify its framing.

\begin{defn}
Let $D\subset S^1\times [-1,1]$ be a $n$-crossing diagram of a link in $S^1\times S^2$. We denote with $s_+$ and $s_-$ the constant states of $D$ assigning respectively always $+1$ and always $-1$. The diagram $D$ is said to be \emph{plus-adequate} if $s_+D > sD$ for every $s$ such that $\sum_i s(i)=n-2$, namely for every state $s$ differing from $s_+$ only at a crossing. It is said to be \emph{minus-adequate} if $s_-D > sD$ for every $s$ such that $\sum_i s(i) = 2-n$. The diagram $D$ is \emph{adequate} if it is both plus-adequate and minus-adequate.

A link diagram is \emph{connected} if it is so as a 4-valent graph.
\end{defn}

\begin{rem}\label{rem:adeq}
If the diagram is contained in a 2-disk, being plus-adequate (resp. minus-adequate) is equivalent to the following fact: in the splitting of the diagram $D$ corresponding to $s_+$ (resp. $s_-$), the two strands replacing a crossing of $D$ lie on two different components of that splitting of $D$ (remark after \cite[Definition 5.2]{Lickorish}). If the diagram is not contained in a 2-disk the two conditions are not related: Fig.~\ref{figure:rem_adeq} shows two diagrams in the annulus satisfying the condition described above, but only the diagram on the right is plus-adequate.
\end{rem}

\begin{figure}
\begin{center}
\includegraphics[scale=0.55]{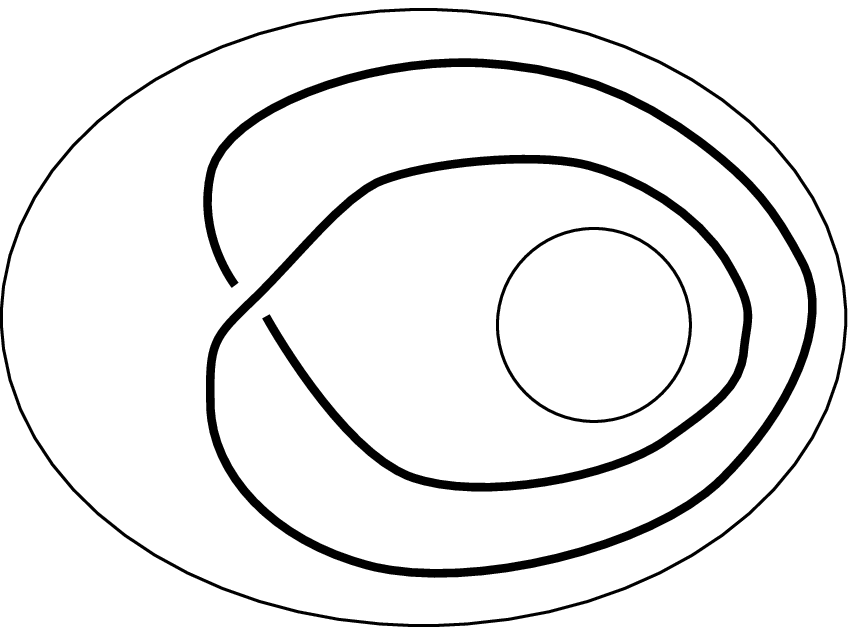} 
\hspace{0.5cm}
\includegraphics[scale=0.55]{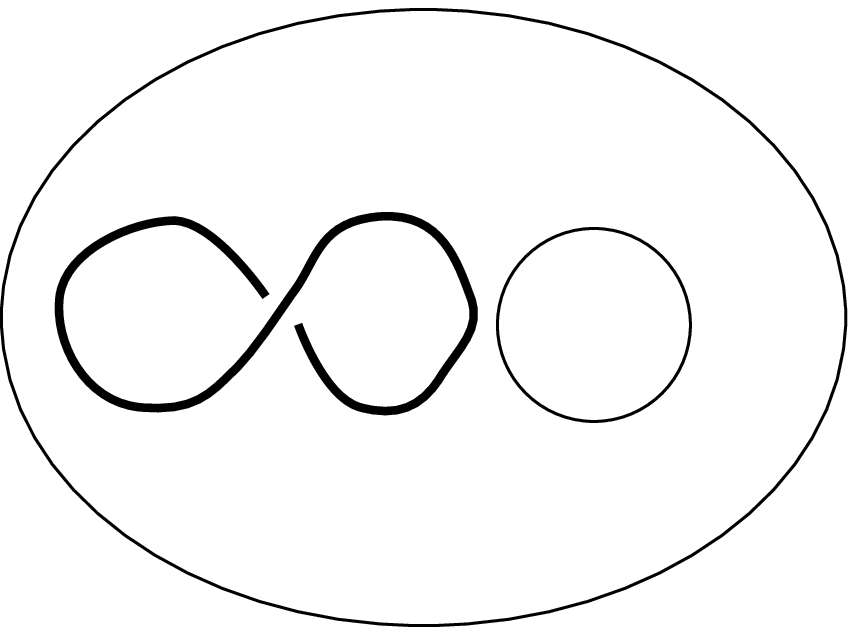}
\end{center}
\caption{Two diagrams satisfying the condition of Remark~\ref{rem:adeq}. Only the one on the right is plus-adequate.}
\label{figure:rem_adeq}
\end{figure}

In the following proposition we show how the technical condition of adequacy is implied by some simple geometric properties.
\begin{prop}\label{prop:reduced_D}
Once an embedding of the annulus is fixed, every simple, alternating, connected diagram of a $\Z_2$-homologically trivial link in $S^1\times S^2$ is adequate.
\begin{proof}
Every link diagram $D\subset S^1\times [-1,1]$ divides the annulus in 2-dimensional connected \emph{regions}. We say that a region is \emph{external} if it touches the boundary of the annulus, otherwise it is \emph{internal}. There are two cases: 
\begin{enumerate}
\item{$D$ is contained in a 2-disk,}
\item{$D$ is not contained in a 2-disk.}
\end{enumerate} 
Since $D$ is connected we have that in the first case the internal regions are disks and there is only one external region that is a punctured annulus. In the second case the internal regions are disks, and the external regions are two annuli. We can give to the regions a black/white coloring as in a chessboard. The link represented by $D$ is $\Z_2$-homologically trivial, hence the external regions are colored in the same way. We say that a boundary component of a region is \emph{internal} if it is different from a boundary component of the annulus $S^1\times \{\pm 1\}$. Since $D$ is alternating the splitting with $s_+$ is equal to the union of the internal boundary components of all the black regions or the white ones (see Fig.~\ref{figure:black-white}). We assume that the components of the splitting with $s_+$ bound the black regions, hence the splitting of $D$ with $s_-$ is equal to the union of the internal boundary components of all the white regions. Changing the splitting of a crossing either merges two different components of the splitting of $D$, or divides one component in two. The diagram $D$ is not plus-adequate if and only if there is a crossing $j$ in $D$ such that after splitting every crossing different from $j$ in the positive way we have a situation as in Fig.~\ref{figure:cases_lem}: if we change the splitting from $+1$ to $-1$ on $j$ we get one more homotopically trivial component by
\begin{enumerate}
\item{dividing in two a previous homotopically trivial component (Fig.~\ref{figure:cases_lem}-(left));}
\item{dividing in two a homotopically non trivial component (Fig.~\ref{figure:cases_lem}-(center));}
\item{fusing two homotopically non trivial components (Fig.~\ref{figure:cases_lem}-(right)).}
\end{enumerate}
In the various cases this implies that
\begin{enumerate}
\item{there is a black region that is adjacent twice to the same crossing;}
\item{the considered homotopically non trivial component must bound an external region that is a black annulus adjacent twice to the same crossing;}
\item{the two considered homotopically trivial components must bound the two external annulus regions and these are black and adjacent to the same crossing.}
\end{enumerate}
The first two cases happen only if the crossing is as in Fig.~\ref{figure:reducedD}. The third case happens only if the crossing is as in Fig.~\ref{figure:reducedD2}. Therefore all these cases are avoided by our hypothesis. In the same way we prove the minus-adequacy. 
\end{proof}
\end{prop}

\begin{figure}[htbp]
$$
\pic{1.6}{0.4}{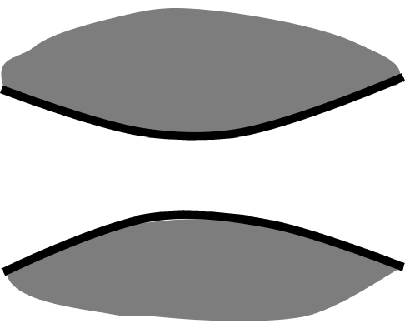} \ \stackrel{+1}{\longleftarrow} \ \pic{1.6}{0.4}{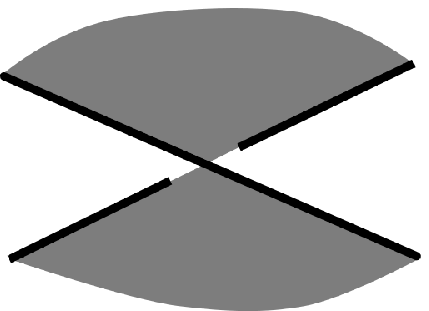} \ \stackrel{-1}{\longrightarrow} \ \pic{1.6}{0.4}{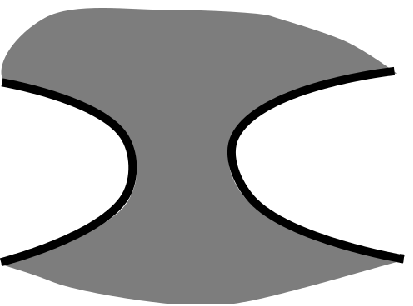}
$$
\caption{The colors of the
regions near a crossing in an alternating diagram are all of this
type. Therefore the splitting with $s_+$ (resp. $s_-$) gives the internal boundary components of all the black (white) regions.}
\label{figure:black-white}
\end{figure}

\begin{figure}[htbp]
\begin{center}
\parbox[c]{3cm}{ 
\includegraphics[scale=0.4]{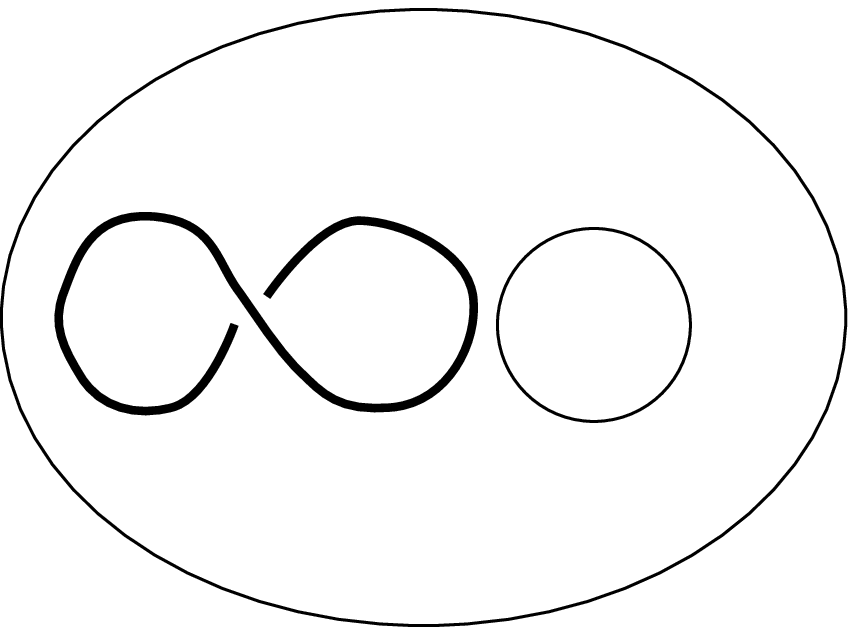} \\
\includegraphics[scale=0.4]{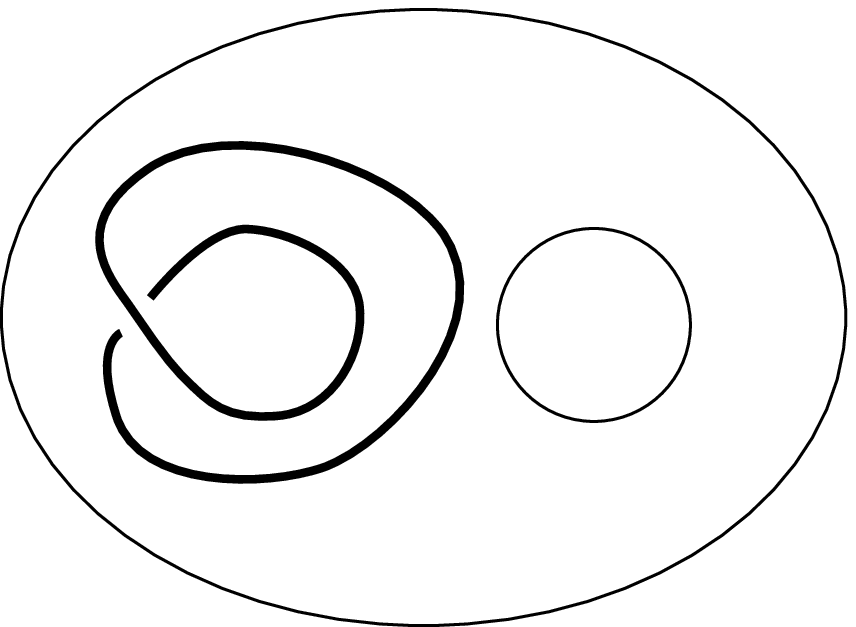} }
\hspace{0.35cm}
\parbox[c]{3cm}{  \includegraphics[scale=0.4]{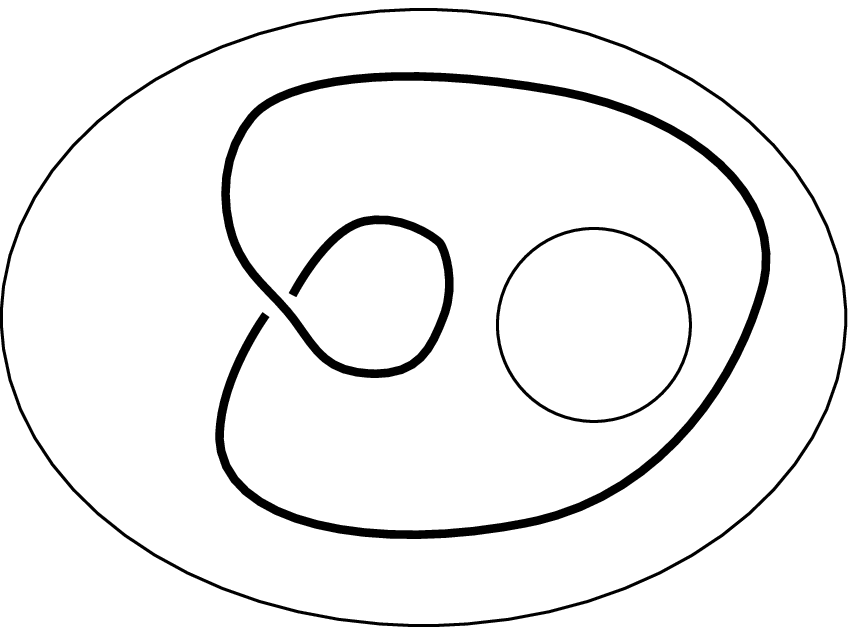} }
\hspace{0.35cm}
\parbox[c]{3cm}{ \includegraphics[scale=0.4]{rem_adeq1.eps} }
\end{center}
\caption{The cases described in the proof of Proposition~\ref{prop:reduced_D} and Lemma~\ref{lem:ineq1}.}
\label{figure:cases_lem}
\end{figure}

Finally we introduce another topological notion that is related to a condition on diagrams. 
\begin{defn}\label{defn:split}
A link $L \subset S^1\times S^2$ is said to be \emph{H-split} if there are two non empty sub-links, $L_1$ and $L_2$, such that $L= L_1\cup L_2$ and they are separated by either a trivial sphere or a Heegaard torus. That means that either there is a 3-ball $B$ such that $L_1\subset B$ and $L_2\subset (S^1\times S^2)\setminus B$, or there are two disjoint solid tori, $V_1$ and $V_2$ such that $S^1\times S^2= V_1\cup V_2$, $L_1\subset V_1$ and $L_2\subset V_2$ ($\partial V_1= \partial V_2$ is the Heegaard torus). 
\end{defn}
This is a natural extension of the usual notion of ``split'' in $S^3$. Clearly a link is H-split if and only if it may be represented by a non connected diagram in the annulus for any embedding of the annulus.

\section{Proof of the theorems}
In this section we prove Theorem~\ref{theorem:Tait_conj} and Theorem~\ref{theorem:Tait_conj_Jones}. We follow the classical proof of Kauffman, K. Murasugi and Thistlethwaite (see for instance \cite[Chapter 5]{Lickorish}) applying some modifications for our case.

\begin{lem}\label{lem:ineq1}
Fix an embedding of the annulus. Let $D\subset S^1\times [-1,1]$ be a $n$-crossing diagram of a $\Z_2$-homologically trivial link in $S^1\times S^2$. Then
$$
B(\langle D \rangle) \leq 2(n + s_+D + s_-D ) .
$$
Furthermore if $D$ is adequate the equality holds:
$$
B(\langle D \rangle) = 2(n + s_+D + s_-D ) .
$$
\begin{proof}
As said in Remark~\ref{rem:p(s)} the quantity $p(s)$ is always even. For every state $s$ we denote with $M(s)$ and $m(s)$ respectively the maximum and the minimum degree of the non zero monomials in $\langle D | s\rangle$ (see Proposition~\ref{prop:state_sum}). Hence
$$
B(\langle D \rangle) \leq \max_s M(s) - \min_s m(s) ,
$$
and the equality holds if there is a unique maximal $s$ for $M$ and a unique minimal $s$ for $m$. We have:
\beq
M(s) & = & \sum_i s(i) + 2sD ,\\
m(s) & = & \sum_i s(i) - 2sD .
\eeq
Let $s$ and $s'$ be two states differing only at a crossing $j$ where $s(j)=+1$ and $s'(j)=-1$. After the splitting of every crossing different from $j$ we have the cases pictured in Fig.~\ref{figure:cases_lem} up to mirror image:
\begin{enumerate}
\item{the crossing $j$ lies on a component of the diagram that is contained in a 2-disk (Fig.~\ref{figure:cases_lem}-(left)), here $s'D = sD \pm 1$ and $p(s')=p(s)$;}
\item{the crossing $j$ lies on a component of the diagram that is not contained in a 2-disk and is $\Z_2$-homologically non trivial (Fig.~\ref{figure:cases_lem}-(center)), here $s'D = sD \pm 1$ and $p(s')= p(s)$;}
\item{the crossing $j$ lies on a component of the diagram that is not contained in a 2-disk and is $\Z_2$-homologically trivial (Fig.~\ref{figure:cases_lem}-(right)), here $s'D = sD \pm 1$ and $p(s')= p(s) \mp 2$.}
\end{enumerate}
In each case we get
\beq
M(s) - M(s') & = & 2 \mp 2\ \geq 0 ,\\
m(s) - m(s') & = & 2 \pm 2 \ \geq 0 .
\eeq
Given a state $s$ we can connect it to $s_+$ finding a sequence of states $s_+=s_0 , s_1 \ldots , s_k =s$ such that $s_r$ differs from $s_{r+1}$ only at a crossing and $\sum_i s_r(i) = \sum_i s_{r+1}(i) +2$. Analogously for $s_-$. Hence $M(s) \leq M(s_+)$ and $m(s) \geq m(s_-)$. Thus we have the first statement. If the diagram is plus-adequate (resp. minus-adequate) it holds $M(s_+) - M(s)=4$ ($m(s) - m(s_-) =4$) for each $s$ such that $\sum_s s(i)= n-2$ ($= 2-n$). Namely we have a strict inequality already from the beginning of the sequence $s_0,\ldots, s_k$. This proves the second statement.
\end{proof}
\end{lem}

\begin{lem}\label{lem:ineq2}
Every $n$-crossing connected diagram $D$ of a $\Z_2$-homologically trivial link in $S^1\times S^2$ satisfies
$$
s_+D + s_- D  \leq \begin{cases}
n+2 & \text{if $D$ is contained in a 2-disk} \\
n & \text{otherwise}
\end{cases} .
$$
\begin{proof}
We proceed by induction on $n$. Suppose that $D$ is contained in a 2-disk. If $n=0$, $D$ can only be a homotopically trivial circle and hence the statement holds. If $D$ can not be contained in a 2-disk it must have at least one crossing. There are only two diagrams with one crossing and not contained in a 2-disk (one is the mirror image of the other): the one in Fig.~\ref{figure:cases_lem}-(right). Their Kauffman bracket are $-A^{\pm 3}$ and the statement holds. Suppose it is true for $n-1$ crossings. We claim (and prove it at the end of this proof) that there is a crossing $j$ of $D$ such that one of the two possible splittings of that crossing produces a diagram $D'$ with the same characteristics of $D$ (connected and contained or not contained in a 2-disk and representing a $\Z_2$-homologically trivial link) with $n-1$ crossings. We can suppose that $D'$ is obtained with a positive splitting on $j$ (the case where $D'$ is obtained splitting in the negative way is analogous). Hence $s_+D' = s_+D$, $ p(s_+,D') = p(s_+)$ ($p(s,D')$ is the quantity $p(s)$ related to the diagram $D'$). Up to mirror image we have the cases pictured in Fig.~\ref{figure:cases_lem}, in all cases $s_-D' = sD \pm 1$. The inductive hypothesis leads to the desired inequality.

Now we prove the claim. One of the two splittings of each crossing must be connected. If $D$ is contained in a 2-disk, each of its splittings is so and we are done. Suppose that $D$ is not contained in a 2-disk. If there are at least two crossings as the ones of Fig.~\ref{figure:reducedD2}, then we take the crossing $i$ as one of them. We get a non connected splitting on $i$ and a connect one that is not contained in a 2-disk. Suppose there is at most one crossing as the ones of Fig.~\ref{figure:reducedD2}. Since $n\geq 2$ we can take the crossing $i$ as one not like the ones Fig.~\ref{figure:reducedD2}, hence both the splittings with $i$ are not contained in a 2-disk. 
\end{proof}
\end{lem}

\begin{lem}\label{lem:alter_eq}
Let $D\subset S^1\times [-1,1]$ be a $n$-crossing, connected and alternating link diagram of a $\Z_2$-homologically trivial link in $S^1\times S^2$. Then
$$
s_+D + s_-D  = \begin{cases}
n+2 & \text{if $D$ is contained in a 2-disk} \\
n & \text{otherwise}
\end{cases} .
$$
\begin{proof}
The number of edges of $D$, seen as a 4-valent graph, is $2n$. The diagram $D$ divides $S^1\times [-1,1]$ in 2-dimensional regions. A count of the Euler characteristic says that the sum of the characteristics of the regions is equal to $n$.

Suppose that $D$ is contained in a 2-disk. Since it is connected all the regions except the external one are disks. The external one is a punctured annulus. Since $D$ is alternating $s_+D + s_-D$ is equal to the number of regions. Hence
$$
s_+D + s_-D  = \text{no. regions }  = \sum_{R \text{ region}} \chi(R) + 2 = n +2 .
$$

Now suppose that the graph $D\subset S^1\times [-1,1]$ is not contained in a 2-disk. Since $D$ is connected all the regions except two are disks. The remaining two (the external ones) are annuli. Since $D$ is alternating $s_+D + p(s_+) + s_-D + p(s_-)$ is equal to the number of regions. Hence
$$
s_+D + p(s_+) + s_-D + p(s_-)  = \text{no. regions}  = \sum_{R \text{ region}} \chi(R) + 2 = n +2 .
$$
There are two regions that are not disks. Since the represented link is $\Z_2$-homogically trivial, in a black and white coloring the external regions must be colored in the same way. Hence we get their boundary either with $s_+$ or with $s_-$. Therefore either $p(s_+)=2$ and $p(s_-)=0$, or $p(s_+)=0$ and $p(s_-)=2$.
\end{proof}
\end{lem}

\begin{theo}\label{theorem:Tait}
Let $D\subset S^1\times [-1,1]$ be a $n$-crossing, connected diagram of a $\Z_2$-homologically trivial link in $S^1\times S^2$. Then
$$
B(\langle D \rangle) \leq \begin{cases}
4n +4 & \text{if $D$ is contained in a 2-disk} \\
4n & \text{otherwise}
\end{cases} .
$$
If $D$ is also alternating and simple, then
$$
B(\langle D \rangle) = \begin{cases}
4n + 4 & \text{if $D$ is contained in a 2-disk} \\
4n & \text{otherwise}
\end{cases} .
$$
\begin{proof}
The first statement follows from the first statement of Lemma~\ref{lem:ineq1} and from Lemma~\ref{lem:ineq2}. The second statement follows from Proposition~\ref{prop:reduced_D}, the second statement of Lemma~\ref{lem:ineq1} and Lemma~\ref{lem:alter_eq}.
\end{proof}
\end{theo}

\begin{proof}[Proof of Theorem~\ref{theorem:Tait_conj}]
Every diagram of $L$ is connected and not contained in a 2-disk because $L$ is non H-split and not contained in a 3-ball. By the second statement of Theorem~\ref{theorem:Tait}, we have $B(\langle D \rangle) = 4n(D)$, where $n(D)$ is the number of crossings of $D$. Let $D'\subset S^1\times [-1,1]$ be another diagram of $L$ for some embedded annulus (maybe different from the one of $D$). By the first statement of Theorem~\ref{theorem:Tait}
$$
4n(D) = B(\langle D\rangle) =B(\langle D' \rangle) \leq 4n(D') .
$$
Therefore $n(D)\leq n(D')$ for every diagram $D'$ of $L$.
\end{proof}

\begin{proof}[Proof of Theorem~\ref{theorem:Tait_conj_Jones}]
In Theorem~\ref{theorem:Tait} we discussed the case of simple diagrams. Thus we suppose that $D$ has no crossings as the ones of Fig.~\ref{figure:reducedD} but has $k>0$ crossings as the ones of Fig.~\ref{figure:reducedD2}. By the second equation in Fig.~\ref{figure:sphere} with $a=b=1$, the Kauffman bracket of $D$ is equal to $1/(-A^2-A^{-2})$ times the bracket of a $n$-crossing diagram $D'$ lying on a 2-disk. Hence
$$
B(\langle D \rangle) = B(\langle D' \rangle) -4 .
$$
It is easy to check that $D'$ is alternating. There are exactly $k$ crossings in $D'$ as the ones of Fig.~\ref{figure:reducedDS3}. We can easily get a reduced alternating $(n-k)$-crossing diagram $D''\subset D^2$ representing the same link of $D'$. Hence by Theorem~\ref{theorem:Tait}
\beq
B(\langle D \rangle) & = & B(\langle D'' \rangle) -4 \\
& = & 4(n-k) +4 -4 \\
& = & 4n- 4k .
\eeq
\end{proof}

\begin{cor}\label{cor:conj_Tait_Jones}
Let $L\subset S^1\times S^2$ be a non H-split $\Z_2$-homologically trivial link. 
\begin{enumerate}
\item{If $B(\langle L \rangle)$ is not a positive multiple of $4$, then either $L$ is the knot in $S^1\times S^2$ with crossing number $1$, or $L$ is not alternating.}
\item{If $L$ is contained in a 3-ball and $B(\langle L \rangle)< 4n+4$, then either $L$ is not alternating, or $L$ has crossing number lower than $n$.}
\item{If $L$ is not contained in a 3-ball, $L$ does not intersect twice a non separating 2-sphere and $B(\langle L \rangle)< 4n$, then either $L$ is not alternating, or $L$ has crossing number lower than $n$.}
\end{enumerate}
\begin{proof}
It follows from Theorem~\ref{theorem:Tait_conj_Jones}.
\end{proof}
\end{cor}

\begin{ex}\label{ex:no_alt}
The knots represented by the diagrams in Fig.~\ref{figure:ex_0Kauf} and in Fig.~\ref{figure:ex_no_alt} are $\Z_2$-homologically trivial and have Kauffman bracket respectively equal to $0$ and $A-A^{-3}-A^{-5}$. Therefore by Corollary~\ref{cor:conj_Tait_Jones}-(1) they are not alternating.
\end{ex}

\begin{figure}
\begin{center}
\includegraphics[scale=0.55]{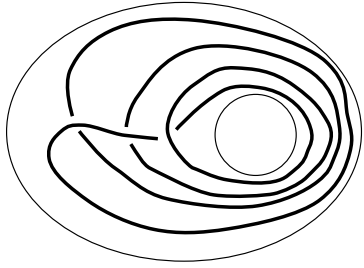} 
\end{center}
\caption{A $\Z_2$-homologically trivial knot in $S^1\times S^2$ whose Kauffman bracket is $A-A^{-3}-A^{-5}$.}
\label{figure:ex_no_alt}
\end{figure}

\section{Open questions}
In this section we expose some conjectures that would extend our result to the case of links in the connected sum of $g\geq 0$ copies of $S^1\times S^2$ and make them more complete. We remind that we already asked two questions before: Question~\ref{quest:reid_moves} and Question~\ref{quest:qred_no_red}.

\begin{defn}
A link diagram in the disk with $g\geq 0$ holes is \emph{simple} if there are no embedded disks in the punctured disk whose boundary intersect the diagram exactly in one crossing (as the diagrams of Fig.~\ref{figure:reducedD}), and there are no crossings adjacent to two external regions (as the ones of Fig.~\ref{figure:reducedD2}) (neither crossings adjacent twice to the same external region).
\end{defn}

\begin{quest}
Does the following version version of Theorem~\ref{theorem:Tait_conj} for links in $\#_g(S^1\times S^2)$ hold for any $g \geq 0$? ``Fix a decomposition of $\#_g(S^1\times S^2)$ as the double of a handlebody and fix an embedding of the disk with $g$ holes in one such handlebody. Let $D$ be a connected, alternating, simple diagram in the disk with $g\geq 0$ holes that represent a $\Z_2$-homologically trivial link in $\#_g(S^1\times S^2)$ whose diagrams in the punctured disk are all connected and not contained in a disk with $g'<g$ holes. Then $D$ has the minimal number of crossings''.
\end{quest}

\begin{quest}
Does the following extension of Theorem~\ref{theorem:Tait_conj_Jones} hold for every $g\geq 0$?: ``Let $D$ be a $n$-crossing, connected, alternating diagram in the disk with $g$ holes that can not be contained in a disk with $g'<g$ holes and has no crossings as the ones of Fig.~\ref{figure:reducedD} and represents a $\Z_2$-homologically trivial link in $\#_g(S^1\times S^2)$. Then
$$
B(\langle D \rangle) = 4n +4 -4g -4k ,
$$ 
where $k$ is the number of crossings adjacent to two external regions (Fig.~\ref{figure:reducedD2})''.
\end{quest}

\begin{quest}
We have considered all the interesting links in $S^1\times S^2$, but in order to make Theorem~\ref{theorem:Tait_conj} really complete we should remove the hypothesis of ``non H-split'' and ``not contained in a 3-ball''. We think that the result remains true even if we remove those hypothesis. In particular we think that a link represented by an alternating connected diagram in some annulus is non H-split (the result is true for links in $S^3$ \cite[Chapter 4]{Lickorish}). Are these conjectures true? Are the natural extension of these conjectures in $\#_g(S^1\times S^2)$ true for any $g\geq 0$?
\end{quest}

\end{document}